\documentclass[11pt]{amsart}
\usepackage{graphicx,verbatim,lineno,titletoc}
\usepackage{amssymb,mathrsfs}
\usepackage{amsmath}
\usepackage{adjustbox}
\usepackage{calc}
\usepackage{array}
\usepackage[font=small]{caption}
\usepackage[T1]{fontenc}
\usepackage{fourier}
\usepackage[english]{babel}
\usepackage{appendix}
\usepackage{enumerate}
\usepackage[english]{babel}
\usepackage{multirow}
\usepackage{float}
\usepackage{enumitem}
\usepackage[numbers]{natbib}
\usepackage{tikz}
\usetikzlibrary{shapes.multipart,patterns,arrows}
\usepackage{hyperref}
\hypersetup{colorlinks,linkcolor={red},citecolor={olive},urlcolor={red}}
\usepackage[top=1.3in, left=1.2in, right=1.2in, bottom=1.5in]{geometry}
\usepackage{mathtools}
\mathtoolsset{showonlyrefs}

\newcolumntype{M}[1]{>{\centering\arraybackslash}m{#1}}
\newcolumntype{N}{@{}m{0pt}@{}}

\pdfoutput=1

\def\.{\hskip.06cm}
\def\ts{\hskip.03cm}

\def\zz{\mathbb Z}
\def\nn{\mathbb N}

\def\rr{\mathbb R}

\def\ov{\widebar}

\def\de{\delta}

\def\<{\langle}
\def\>{\rangle}

\def\rH{{\text {\rm H} } }

\def\0{{\mathbf 0}}

\def\nN{{\mathcal N}}

\def\.{\hskip.06cm}
\def\ts{\hskip.03cm}

\def\P{{\textup{\textsf{P}}}}

\def\liminf{\mathop{\rm lim\,inf}\limits}
\def\limsup{\mathop{\rm lim\,sup}\limits}

\def\R{\mathbb{R}}

\def\E{\mathbb{E}}
\def\P{\mathbb{P}}

\def\r{\mathtt{r}}

\def\eps{\varepsilon}

\def\r{\mathbf{r}}
\def\c{\mathbf{c}}

\def\M{\mathcal{M}}

\def\Geom{\textup{Geom}}

\makeatletter
\newcommand*\rel@kern[1]{\kern#1\dimexpr\macc@kerna}
\newcommand*\widebar[1]{%
	\begingroup
	\def\mathaccent##1##2{%
		\rel@kern{0.8}%
		\overline{\rel@kern{-0.8}\macc@nucleus\rel@kern{0.2}}%
		\rel@kern{-0.2}%
	}%
	\macc@depth\@ne
	\let\math@bgroup\@empty \let\math@egroup\macc@set@skewchar
	\mathsurround\z@ \frozen@everymath{\mathgroup\macc@group\relax}%
	\macc@set@skewchar\relax
	\let\mathaccentV\macc@nested@a
	\macc@nested@a\relax111{#1}%
	\endgroup
}

\DeclareMathOperator*{\argmax}{arg\,max}

\usepackage{theoremref}

\newtheorem{theorem}{Theorem}
\numberwithin{theorem}{section}
\numberwithin{equation}{section}

\newtheorem{lemma}[theorem]{Lemma}
\newtheorem{prop}[theorem]{Proposition}
\newtheorem{corollary}[theorem]{Corollary}
\newtheorem{conjecture}[theorem]{Conjecture}

\theoremstyle{definition}
\newtheorem{definition}[theorem]{Definition}

\newtheorem{remark}[theorem]{Remark}

\allowdisplaybreaks

\makeatletter
\newcommand{\addresseshere}{%
	\enddoc@text\let\enddoc@text\relax
}
\makeatother

\begin{document}
	
	\title[Phase transition in random contingency tables]{Phase transition in random contingency tables \\ with non-uniform margins}

	\author{Sam Dittmer}
	\address{Sam Dittmer, Department of Mathematics, UCLA, Los Angeles, CA 90095, USA}
	\email{\texttt{samuel.dittmer@math.ucla.edu}}

	\author{Hanbaek Lyu}
	\address{Hanbaek Lyu, Department of Mathematics, UCLA, Los Angeles, CA 90095, USA}
	\email{\texttt{hlyu@math.ucla.edu}}

	\author{Igor Pak}
	\address{Igor Pak, Department of Mathematics, UCLA, Los Angeles, CA 90095, USA}
	\email{\texttt{pak@math.ucla.edu}}


\begin{abstract}
For parameters $n,\delta,B,$ and $C$, let $X=(X_{k\ell})$ be the random uniform contingency table whose first $\lfloor n^{\delta} \rfloor $ rows and columns have margin \ts $\lfloor BCn \rfloor$ \ts and the last $n$ rows and columns have margin \ts $\lfloor Cn \rfloor$. For every \ts $0<\delta<1$, we establish a sharp phase transition of the limiting distribution of each entry of $X$ at the critical value \ts $B_{c}=1+\sqrt{1+1/C}$. In particular, for \ts $1/2<\delta<1$, we show that the distribution of each entry converges to a geometric distribution in total variation distance, whose mean depends sensitively on whether \ts $B<B_{c}$ \ts or \ts $B>B_{c}$. Our main result shows that \ts $\E[X_{11}]$ \ts is uniformly bounded for \ts $B<B_{c}$, but has sharp asymptotic \ts $C(B-B_{c})\ts n^{1-\delta}$ \ts for \ts $B>B_{c}$. We also establish a strong law of large numbers for the row sums in top right and top left blocks.
\end{abstract}
	
	${}$
	\vspace{-1.70cm}
	${}$
	\maketitle

	\section{Introduction}
	\label{Introduction}
	
	\subsection{Random contingency tables}
	Contingency tables are fundamental objects in statistics for studying
	dependence structure between two or more variables, see e.g.~\cite{Ev,FLL,Kat}.
	They also correspond to bipartite multi-graphs with given degrees and play
	an important role in combinatorics and graph theory, see e.g.~\cite{B1,DG,DS}.
	Random contingency tables have been intensely studied in a variety of
	regimes, yet remain largely out of reach in many interesting
	special cases, see e.g.~\cite{B3,CM}.
	
	Let $\r=(r_1,\ldots,r_m) \in \nn^{m}, \c=(c_1,\ldots,c_n)\in \nn^{n}$ be two
	nonnegative integer vectors with the same sum of entries.
	Denote by $\M(\r,\c)$  the set of all $(n\times m)$ contingency tables with
	row sums~$r_i$ and column sums~$c_j$, i.e.
	\begin{align}\label{eq:def_contingency_set}
	\M(\r,\c) \, := \, \left\{ \. \bigl(a_{ij}\bigr)\in \nn^{mn} \,\.\bigg|\,\. \sum_{k=1}^{n}a_{ik} = r_i, \,
	\sum_{k=1}^{m} a_{kj} = c_j\, \ \, \text{for all} \ \. 1\le i\le n,\, 1\le j \le m \. \right\}\ts.
	\end{align}
	
	Let $X=(X_{ij})$ be the contingency table chosen uniformly at random from $\M(\r,\c)$.  The
	asymptotic properties of the entries of $X$ as $m,n \to \infty$ is the subject of this paper.
	When the margins are uniform, i.e.\ $r_1=\ldots=r_m$ and $c_1=\ldots=c_n$, the exact asymptotics
	for $\bigl|\M(\r,\c)\bigr|$ are known~\cite{CM,GM}.  In fact, the distribution of individual
	entries~$X_{ij}$ is asymptotically geometric and the dependence between the entries vanish as the size
	of the table goes to infinity~\cite{CDS}.
	
	In this paper we analyze random square contingency table
	with the first $\lfloor n^{\delta} \rfloor$ row and column margins $\lfloor BCn \rfloor$,
	and the last $n$ row and column margins $\lfloor Cn \rfloor$. Viewing such $X$ as
	block matrices, see Figure~\ref{fig:CT}, it is natural to assume that the entries are again
	nearly independent and identically distributed within each block.  However, there is still
	one degree of freedom remaining: the distribution of mass of each block.  We establish a sharp
	phase transition for this distribution.  The following corollary is a special case of general
	results we present in the next section.
	
	\begin{figure}[hbt]
		\begin{center}
			\includegraphics[width = 0.38 \textwidth]{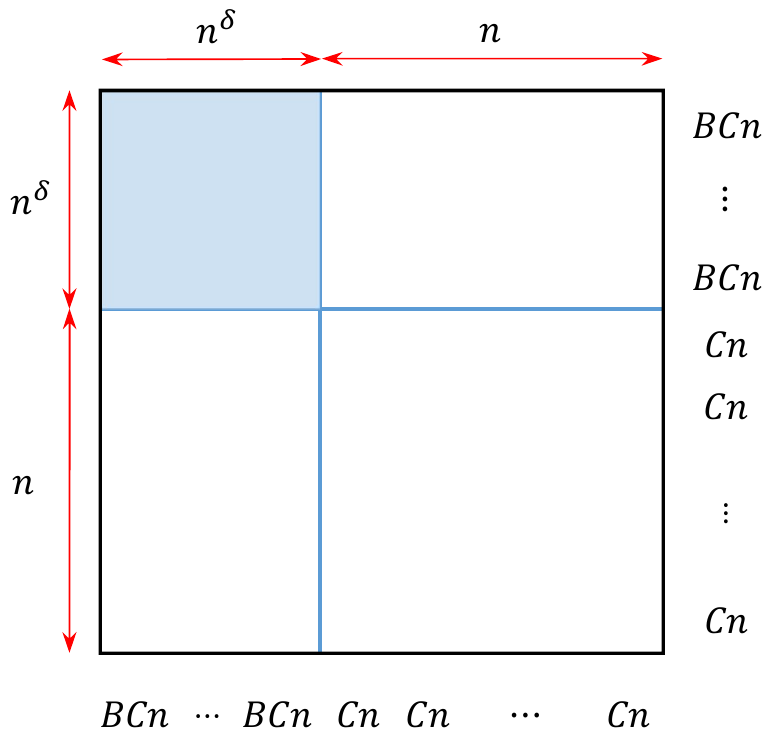}
		\end{center}
		\caption{ Contingency table with parameters $n,\delta,B$ and $C$.
			First $\lfloor n^{\delta} \rfloor$ rows and columns have margins $\lfloor Cn \rfloor$,
			the last $n$ rows and columns have margins $\lfloor BCn \rfloor$.
		}
		\label{fig:CT}
	\end{figure}
	
	\begin{corollary} [see Theorem~\ref{thm:main_expectation}]
		Fix constants $B, C> 0 $ and $1/2<\de<1$.
		Let $X=(X_{ij})$ be the uniform random contingency table
		with the first $\lfloor n^{\delta} \rfloor$ row and column margins $\lfloor BCn \rfloor$,
		and the last $n$ row and column margins $\lfloor Cn \rfloor$.  Then:
		$$
		\E[X_{11}] \. = \. \Theta(1) \ \ \, \text{for} \ \ B<B_c\ts, \quad \, \text{and} \, \quad
		\E[X_{11}] \. = \. \Theta\bigl(n^{1-\de}\bigr) \ \ \, \text{for} \ \ B>B_c\ts,
		$$
		where the \emph{critical value} \ts $B_{c}\ts =\ts 1+\sqrt{1+1/C}$.
	\end{corollary}
	
	In the next section we present various extensions and refinements of this
	result, including determining the precise constants implied by the $\Theta$ notation.
	We also extend the results to $0<\de \le 1/2$, although our results are not
	as strong in this case.
	
	The story behind the phase transition in the corollary is quite interesting.
	For $\de=0$, $B=3$ and $C=1$ the phenomenon of large $X_{11}$ entry was first observed
	by Barvinok in~\cite[$\S$1.5]{B3} for the (non-uniform) distribution of \emph{``typical''
		contingency tables}.
	In fact, Barvinok showed that for $\de=0$ and $C=1$, the phase transition for
	typical tables happens at  \ts  $B_c=1+\sqrt{2}$\ts, see~\cite{B3,B4}.
	
	In~\cite{DP}, the authors tested empirically uniform contingency tables using
	a new MCMC algorithm introduced in~\cite{DP0}, and the experiments seem to confirm
	Barvinok's conjectured
	value for the critical~$B_c$.  In fact, the simulations show drastically different
	behavior for the subcritical $B<B_c$ vs.\ supercritical $B>B_c$ cases.  Here we analyze
	the $X_{11}$ entry in a random uniform $n\times n$ contingency table with both
	margins \ts $(B\ts n,n,\ldots,n)$, see Figure~\ref{fig:Barv-sim}.  This is the
	case $C=1$, where \ts $B_c=1+\sqrt{2} \approx 2.41$.
\begin{figure}[hbt]
		\begin{center}
			\includegraphics[width = 0.4 \textwidth]{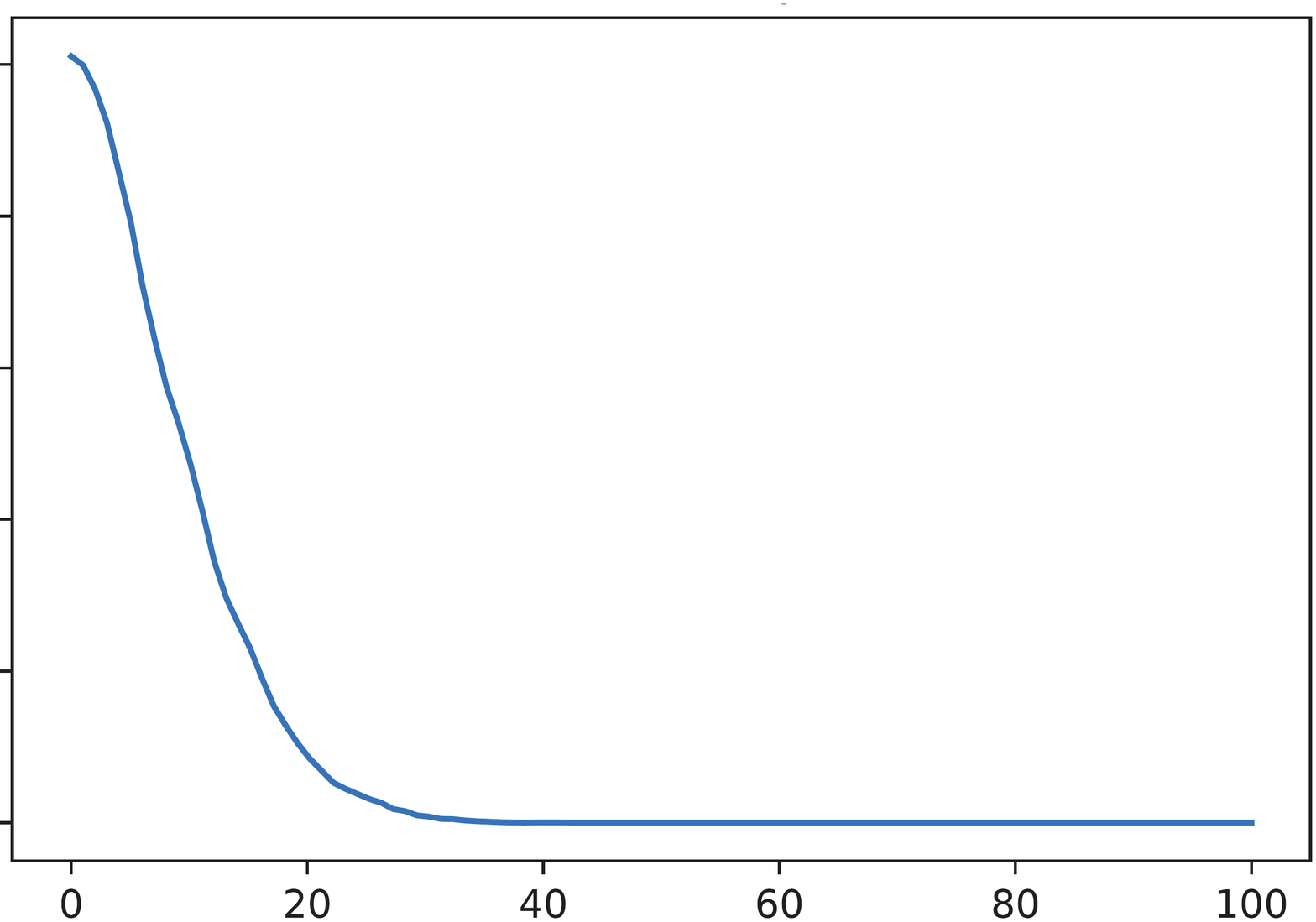} \qquad \ \
			\includegraphics[width = 0.4 \textwidth]{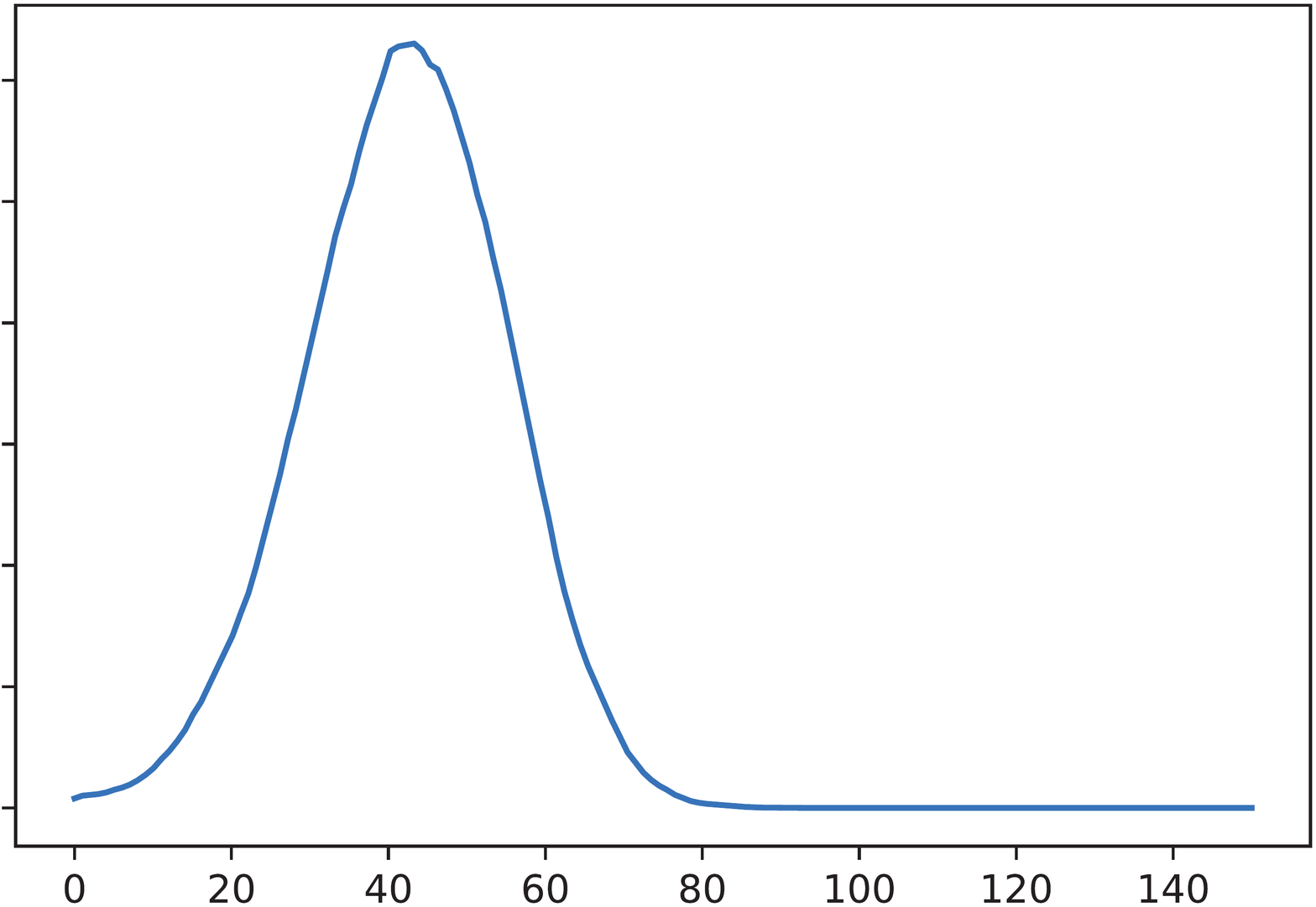}	
		\end{center}
		\caption{Summary of $10^4$ simulations of $X_{11}$ entry in a random uniform $n\times n$ contingency table $X=(X_{ij})$
			with both margins \ts $(B\ts n,n,\ldots,n)$, where $n=50$.  The left graph corresponds to a subcritical value $B=2$,
			and the right graph to a supercritical value $B=3$.}
		\label{fig:Barv-sim}
\end{figure}
%
%
In some sense, the simulations
	showed a sharper phase transition than the typical matrices: not only the expectation
	\ts $\E[X_{11}]$ \ts exhibited jump from bounded to having linear growth, but the
	distribution of $X_{11}$ switches from geometric in the subcritical case to
	normal in the supercritical case.

	This paper gives the first rigorous proof of the
	phase transition in the uniform case.  Although we do not cover the $\de=0$ case
	introduced by Barvinok, we conjecture the phase transition extends to this case.
	In fact, our results go beyond what the simulations in~\cite{DP} suggest, as we
	interpolate between $\de=0$ and (uniform) $\de=1$ case.  Rather surprisingly,
	we show that for $\de>1/2$ the behavior of random uniform and typical matrices
	remains similar, with a geometric distribution in the supercritical case.
	We conjecture that there is an additional phase transition at $\de=1/2$, and
	for $\de<1/2$ the distribution of $X_{11}$ is normal in the supercritical case
    (see Conjecture~\ref{conj:Gaussian}).
	In the limiting case $\de=0$, this is supported by the simulations mentioned above.
	Further conjectures with more refined estimates are given in Section~\ref{s:conj}.
	
	\medskip
	
	\subsection{Background}
	
	Let $\mathcal{P}(\r,\c)\subseteq \R_+^{mn}$ be the \textit{transportation polytope}
	of \emph{real} nonnegative contingency tables with margins $\r$ and~$\c$, i.e.\ defined by~\eqref{eq:def_contingency_set}
	over~$\rr_+$.  Clearly, $\M(\r,\c) = \mathcal{P}(\r,\c) \cap \zz^{mn}$.
	When $r_1=\ldots=r_n=c_1=\ldots=1$, $m=n$, we have $\mathcal{P}(\r,\c)$ is the classical \emph{Birkhoff polytope},
	of interest in Combinatorics, Discrete Geometry, Combinatorial Optimization and Discrete Probability,
	see e.g.~\cite{DK,Pak}. The asymptotic behavior of the \ts vol$\bigl(\mathcal{P}(\r,\c)\bigr)$ \ts
	is known~\cite{CM-birk},  as well as the exact value and the whole Ehrhart polynomial for $n\le 10$,
	see~\cite{BP}, and numerical estimates for $11\le n\le 14$~\cite{CV}. Such sharp volume estimates were crucially used in~\cite{CDS} to analyze the asymptotic behavior of random contingency table~$X$ for uniform margins.
	
	For non-uniform margins the existing sharp asymptotic results cover only \emph{smooth margins}, a technical condition which
	includes the case when all ratios $r_i/r_k$ and $c_j/c_\ell$ are bounded, see~\cite{BLSY,BBK,BC,CM} for precise statements. 
For general margins $\r$ and~$\c$, the upper and lower bounds on $|\mathcal{M}(\r,\c)|$ were given in \cite{B1, B3} (see also Theorem~\ref{thm:Barvinok_summary}). The proofs of our main results rely heavily on these bounds.

	Now, in statistics, a popular practice is to sample $Y$ from the \emph{hypergeometric} (Fisher--Yates) \emph{distribution}
	defined as follows:
	\begin{align}\label{eq:Fisher_Yates}
	\P\bigl(Y=(y_{ij})\bigr) \, := \, \phi(\r,\c) \ts\cdot\. \prod_{i=1}^{m}\.\prod_{j=1}^{n} \. \frac{1}{y_{ij}!}\,,
	\quad  \text{where} \quad\phi(\r,\c)\, = \, \frac{1}{N!} \, \prod_{i=1}^{m} \. r_i! \, \prod_{j=1}^{n} \. c_j!
	\end{align}
	and \ts $N=r_1+\ldots+r_m=c_1+\ldots+c_n$ denotes the total sum of a contingency table in $\M(\r,\c)$.
	We refer to~\cite{DE,Ev} for an extensive discussion and to~\cite{FLL,Kat} for the recent treatment.
	
	The rationale behind this approach lies in the \textit{independence table} \ts
	$W=(w_{ij})\in \mathcal{P}(\r,\c)$, defined by \ts
	$w_{ij} := (r_i\ts c_j)/N$.  This table~$W$ gives both the expectation of the
	Fisher--Yates distribution and is also the unique maximizer of the following
	strictly concave function
	\begin{align}\label{eq:def_H_population_entropy}
	\rH(W) \. = \, \sum_{i,j} \. \frac{w_{ij}}{N} \. \log \frac{N}{w_{ij}}
	\end{align}
	in the transportation polytope $\mathcal{P}(\r,\c)$, see~\cite[Ex.~(iv)]{Good}.
	Note that for each $Q\in \mathcal{P}(\r,\c)$, if we view $Q/N$ as the `population contingency table',
	where each entry $q_{ij}/N$ is understood as the marginal probability of the entry~$(i,j)$,
	then $\rH(W)$ defined at~\eqref{eq:def_H_population_entropy} is the entropy of the probability
	mass function~$Q/N$. However, it is known that the hypergeometric distribution may not properly
	capture the behavior of the uniform contingency table $X\in \mathcal{M}(\r,\c)$.
	
	When one tries to find the marginal distribution for $Y_{ij}$ that maximizes the overall entropy subject to the margin condition, instead of viewing each contingency table as a rescaled probability mass function, one finds that the entries must be independent and geometrically distributed. Furthermore, one can further maximize the entropy by optimizing the mean of each entry. This leads to the notion of the \textit{typical table} $Z=(z_{ij})\in \mathcal{P}(\r,\c)$ (see Definition~\ref{def:typical_table}), introduced by Barvinok in~\cite{B1} and further exploited in~\cite{B2,B3,BH}. The behavior of typical and independence tables are known to be similar when the margins are relatively uniform~\cite{BLSY}, but could be drastically different when the margin are strongly asymmetric~\cite[$\S$1.6]{B3}.
	
	In~\cite{B3}, Barvinok showed that there exists a phase transition in the behavior of the typical table for a simple model of contingency tables with asymmetric margins. Namely, let $Z=(z_{ij})$ be the typical table for $\M(\r,\c)$ where \ts
	$\r=\c \ts := \ts (\lfloor BCn\rfloor, \lfloor Cn\rfloor,\ldots,\lfloor Cn\rfloor)\in \nn^{n+1}$. For $B=1$, all entries of $Z$ are equal by the symmetry. In particular, the corner entry $z_{11}$ is bounded by $C$ for all $n$. On the other hand, Barvinok~\cite{B3} showed that for $B>1+\sqrt{2}$, the entry $z_{11}$ has linear growth
	\begin{align}
	z_{11} \. \ge \. (B-1-\sqrt{2})\ts n \quad \text{for all \ \. $n\ge 1$},
	\end{align}
	while all the other entries of $Z$ are uniformly bounded by \ts $1+\sqrt{2}$. Hence, as~$B$ passes a certain critical value $\le 1+\sqrt{2}$, the ``mass'' within the typical table $Z$ suddenly concentrates at the corner entry~$z_{11}$.  As we mentioned earlier, this was the starting
	observation of the paper.
	

	\subsection{Notation}
	
	We use \ts $\nn=\{0,1,2,\ldots\}$ \ts and \ts $\rr_+ = \{x\in \rr, \ts x\ge 0\}$.
	For all \ts $a,b\in \mathbb{R}$, denote \ts $a\land b:=\min(a,b)$ \ts and \ts $a^{+}:=\max(a,0)$.
	
	For all $\lambda>0$, we write $Y\sim \Geom(\lambda)$  for a discrete random variable $Y$
	with probability mass function\footnote{This notation is somewhat nonstandard, but
		is more convenient for our purposes.}
	\begin{align}
	\hspace{2cm} \P\bigl(Y=k\bigr) \. := \. \left( \frac{1}{1+\lambda} \right)
	\left( \frac{\lambda}{1+\lambda} \right)^{k}\.,\quad \text{for all} \ \ k\in \nn\ts.
	\end{align}
	Note that $\E[Y]=\lambda$. We call $Y$ a \emph{geometric random variable}
	with mean~$\lambda$. For every two probability distributions $\mu_{1}$, $\mu_{2}$
	over a countable sample space $\Omega$, the \textit{total variation distance} is defined as
	\begin{align}
	d_{TV}(\mu_{1},\mu_{2}) \. := \,\. \sum_{x\in \Omega} \, \bigl|\mu_{1}(x) - \mu_{2}(x)\bigr|\ts.
	\end{align}
	Let $X$ and $Y$ be random variables with distribution $\mu_{1}$ and $\mu_{2}$, respectively.
	To simplify the notation, we write:
	\begin{align}
	d_{TV}(\mu_{1},\mu_{2}) \. = \. d_{TV}(X,\mu_{2}) \. = \. d_{TV}(\mu_{1},Y) \. = \. d_{TV}(X,Y)\ts.
	\end{align}
	
	\vspace{0.2cm}
	
	\section{Statement of results}
	
	For parameters $n\ge 1$, $0\le \delta \le 1$, and $B,C\ge 0$, let $\mathcal{M}_{n,\delta}(B,C)=\mathcal{M}(\r,\c)$, where
	\begin{align}\label{eq:def_barvinok_margin}
	\r\. = \. \c \. := \. \bigl(\lfloor BCn \rfloor ,\ldots,\lfloor BCn\rfloor ,\lfloor Cn \rfloor, \ldots, \lfloor Cn \rfloor \bigr)
	\, \in \, \nn^{\lfloor n^{\delta} \rfloor \ts +\ts n}\ts.
	\end{align}
	In other words, $\mathcal{M}_{n,\delta}(B,C)$ \ts is the set of contingency tables whose first $\lfloor n^{\delta} \rfloor $ rows and columns have margin $\lfloor BCn \rfloor$ and the other $n$ rows and columns have margin $\lfloor Cn \rfloor$, see Figure~\ref{fig:CT}. Let $X=(X_{ij})$ be the random contingency table sampled uniformly from $\mathcal{M}_{n,\delta}(B,C)$. We are interested in the asymptotic behavior of the entry $X_{ij}$ as $n\rightarrow \infty$ for various choice of parameters \ts $\delta,\ts B$ and~$C$. Note that the entries within each of the four blocks in Figure~\ref{fig:CT} have the same distribution by the symmetry. Hence it is sufficient to restrict the attention to the entries $X_{11}$, $X_{1,n+1}$, and $X_{n+1,n+1}$.
	
	We establish a sharp phase transition at
	$$B_{c}:=1+\sqrt{1+1/C}$$
	for the limiting expectation of the entries of~$X$. The following theorem shows that the limiting distribution of each entry of $X$ is geometric with mean depending on whether $B<B_{c}$ or $B>B_{c}$; see Figure~\ref{fig:CT_phase}.
	
	\begin{theorem}\label{thm:main_geo}
		Fix constants $\delta,B,C>0$, and let $X=(X_{ij})$ be sampled from $\M_{n,\delta}(B,C)$ uniformly at random.
		Fix $\eps>0$ and let $B_c$ be as above.
		\begin{description}
			\item[(i)] {\text{\rm [bottom right]}} \. For all $0\le \delta\le 1$ and $B,C>0$, we have:
			\begin{align}
			d_{TV}\bigl(X_{n+1,n+1}, \Geom(C)\bigr) \. = \. O\bigl(n^{\delta-1} + n^{-1/2+\eps}\bigr).
			\end{align}
			
			\item[(ii)] {\text{\rm [sides]}} \. For all $0<\delta\le 1$ and $B\ne B_{c}$,  we have:
			\begin{align}
			d_{TV}\bigl(X_{1,n+1}, \Geom(R)\bigr) \. = \.  O\bigl(n^{\delta-1} + n^{-\delta/2+\eps}\bigr),
			\end{align}
			where
			\begin{align}
			R \. = \. \begin{cases}
			BC & \text{\, if \. $B<B_{c}$\,,} \\
			B_{c}C & \text{\, if \. $B>B_{c}$\,.}
			\end{cases}
			\end{align}
			
			\item[(iii)] {\text{\rm [top left]}} \.	For all $1/2<\delta< 1$ and $B\ne B_{c}$, we have:
			\begin{align}
			d_{TV}\bigl(X_{11}, \Geom(R)\bigr)  \. = \.   O\bigl(n^{\delta-1} + n^{1/2 - \delta + \eps}\bigr),
			\end{align}
			where
			\begin{align}
			R \. = \.
			\begin{cases}
			\frac{B^{2}(1+C)}{(B_{c}-B)(B+B_{c}-2)} & \text{\, if \. $B<B_{c}$\,,} \\
			C(B-B_{c})n^{1-\delta} + O(1) & \text{\, if \. $B>B_{c}$\,.}
			\end{cases}
			\end{align}
		\end{description}
	\end{theorem}

	\begin{figure}[hbt]
		\begin{center}
			\includegraphics[width = 0.7 \textwidth]{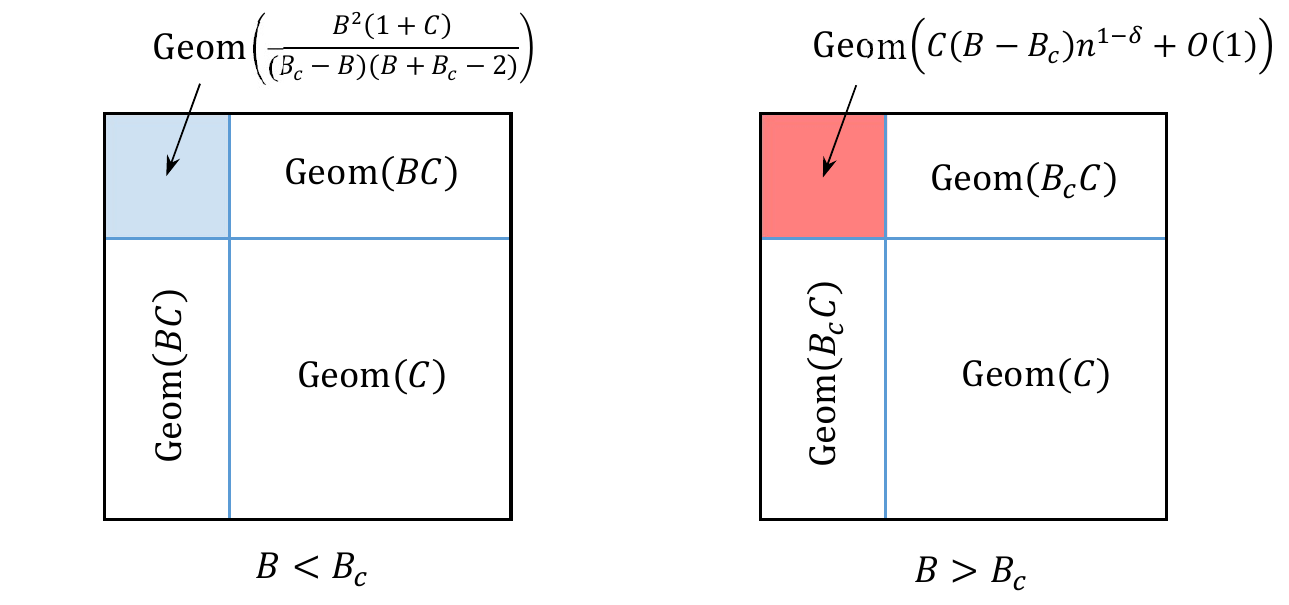}
		\end{center}
		\caption{ Limiting distributions of the entries in the uniform contingency table $X\in \mathcal{M}_{n,\delta}(B,C)$ in the subcritical $B<B_{c}$ (left) and supercritical $B>B_{c}$ (right) regimes for thick bezels $1/2<\delta<1$.
		}
		\label{fig:CT_phase}
	\end{figure}

	Our second result proves the phase transition of entries of random contingency tables $X\in \mathcal{M}_{n,\delta}(B,C)$ in expectation.

	\begin{theorem}\label{thm:main_expectation}
		Fix constants $\delta,B,C>0$, and let \ts $X=(X_{ij})$ \ts be sampled from \ts $\M_{n,\delta}(B,C)$ \ts
		uniformly at random. Let $B_c$ be as above.
		\begin{description}
			\item[(i)] {\text{\rm [bottom right]}} \.  For all $0\le \delta\le 1$, $B,C>0$, and $\eps>0$, we have:
			\begin{align}
			\bigl|\E[X_{n+1,n+1}] - C\bigr| \. = \. O(n^{\delta-1+\eps})\ts.
			\end{align}
			
			\item[(ii)] {\text{\rm [sides]}} \. For all $\alpha,\eps>0$, we have:
			\begin{align}\label{eq:thm2_ex_2_1}
			\begin{cases}
			\bigl| \E[X_{1,n+1}] - BC\bigr| \. = \. O\bigl(n^{\delta-1} + n^{-(\delta/2)+\eps}\bigr) & \text{\, if \. $B<B_{c}$ and $0<\delta<1$}\ts, \\
			\bigl|\E[X_{1,n+1}\land n^{\alpha}]  \. - \. B_{c}C\bigr|\. = \. O\bigl(n^{\delta-1} + n^{\alpha-(\delta/2)+\eps}\bigr) & \text{\, if \. $B>B_{c}$ and $0<\delta<1$}\ts, \\
			\bigl| \E[X_{1,n+1}] \. - \. B_{c}C \bigr| \. = \. O\bigl(n^{(1/2)-\delta+\eps}\bigr)  & \text{\, if \. $B>B_{c}$ \. and \. $1/2<\delta<1$}\ts.
			\end{cases}
			\end{align}
			
			\item[(iii)] {\text{\rm [top left]}} \.	For all $\alpha,\eps>0$, we have:
			\begin{align}\label{eq:thm2_ex_2_2}
			\begin{cases}
			\bigl| n^{\delta-1} \E[X_{11}] \bigr|  \. = \. O\bigl(n^{\delta-1} + n^{-\delta+\eps}\bigr) & \text{\, if \. $B<B_{c}$ \. and \. $0<\delta<1$}\ts, \\
			\left| \E[X_{11}\land n^{\alpha}] \. -  \. \frac{B^{2}(1+C)}{(B_{c}-B)(B+B_{c}-2)}\right|  \. = \. O\bigl(n^{\alpha + (1/2)- \delta + \eps}\bigr)  & \text{\, if \. $B<B_{c}$ \. and \. $1/2<\delta<1$\ts,}
			\end{cases}
			\end{align}
			and
			\begin{align}\label{eq:thm_E_iii}
			\begin{cases}
			\left|C(B-B_{c}) \ts - \ts n^{\delta-1} \E[X_{11}] \ts - \ts \E\bigl[(X_{1,n+1} -  n^{\alpha})^{+}\bigr]\ts\right|
			\. = \. O\bigl(n^{\delta-1} + n^{\alpha-(\delta/2)+\eps}\bigr) & \text{\, if \. $B>B_{c}$ \. and \. $0<\delta<1$}\ts,\\
			\left|  C(B-B_{c}) \ts - \ts n^{\delta-1} \E[X_{11}]\right|  \.=\. O\bigl(n^{(1/2)-\delta+\eps}\bigr) & \text{\,
				if \. $B>B_{c}$ \. and \. $1/2<\delta<1$}\ts.
			\end{cases}
			\end{align}
		\end{description}
	\end{theorem}
	
	\smallskip
	
	Lastly, we establish strong law of large numbers for row sums of entries in $X$ in the top right and bottom right blocks.
	
	\begin{theorem}\label{thm:SLLN}
		Fix $B,C>0 $ and $1/2< \delta<1$. Let $X=(X_{ij})$ be sampled
		from $\M_{n,\delta}(B,C)$ uniformly at random, and let $B_c$ be as above.
		Then a.s., as $n\rightarrow \infty$, we have:	
		\begin{align}
		\lim_{n\rightarrow \infty} \. \frac{1}{n} \. \sum_{k=1}^{n} \. X_{1,k+\lfloor n^{\delta} \rfloor}  \, = \,
		\begin{cases}
		BC & \text{\, if \. $B<B_{c}$}\ts, \\
		B_{c}C & \text{\, if \. $B>B_{c}$}\ts.
		\end{cases}
		\end{align}
		Furthermore, for all \ts $B,C> 0$ \ts and \ts $0\le \delta< 1$, a.s.\ as $n\rightarrow \infty$, we have:
		\begin{align}
		\lim_{n\rightarrow \infty} \. \frac{1}{n} \. \sum_{k=1}^{n} \. X_{n+1,k+\lfloor n^{\delta} \rfloor}  \. = \. C\ts.
		\end{align}
	\end{theorem}
	
	As we mentioned in the introduction, the proofs utilize Barvinok's technology of ``typical tables" and upper and lower bounds on the number $|\mathcal{M}(\r,\c)|$ of contingency tables for general margins.  Except for the next section where we
	summarize conjectural extensions of our theorems, the rest of the paper is dedicated to
	proofs of the theorems.
	
	\vspace{0.2cm}
	\section{Conjectures}\label{s:conj}
	
	In view of the strong law of large numbers for the top right block of $X$ given by Theorem~\ref{thm:SLLN}, we conjecture that a central limit theorem also holds for the top right block in the supercritical regime $B>B_{c}$ for at least when $1/2<\delta<1$. However, we believe that a central limit behavior should not be expected for the subcritical regime $B<B_{c}$. Our rationale is that, according to Theorem~\ref{thm:SLLN}, the first row sum in the top right block of $X$ is asymptotically $BCn$ for $B<B_{c}$, which is the full row sum of $X$. Hence there is not much room for each entry in the top right block to fluctuate. On the other hand, for $B>B_{c}$, the row sum in the top right block contributes only to $B_{c}Cn$, which is independent of $B$. Hence when $B\gg B_{c}$ is large, there is enough room for them to fluctuate, and they would not feel the `bar' of $BC$ since they must fluctuate around $B_{c}C\ll BC$.
	
	\begin{conjecture}\label{conj:CLT}
		Fix $B,C>0$ and $0< \delta<1$. Let $X=(X_{ij})$ be sampled from $\M_{n,\delta}(B,C)$ uniformly at random.  Denote
		\begin{align}
		S_{n,\delta}(B,C) \, := \,  \sum_{k=1}^{n} \. X_{1,k+\lfloor n^{\delta} \rfloor}\..
		\end{align}
		Then as $n\rightarrow \infty$, we have:
		\begin{align}
		\begin{cases}
		\frac1{\sqrt{n}} \.\Bigl(S_{n,\delta}(B,C) - BCn\Bigr) \. \longrightarrow \. 0 \qquad \text{a.s.} & \text{\, if \. $B<B_{c}$}\ts, \\
		\frac1{\sqrt{n}}\. \Bigl(S_{n,\delta}(B,C)-B_{c}Cn\Bigr) \. \Longrightarrow \. \sqrt{B_{c}C + (B_{c}C)^{2}} \ts \cdot \ts \nN(0,1) & \text{\, if \. $B>B_{c}$}\ts,
		\end{cases}
		\end{align}
		where $\nN(0,1)$ is the standard normal distribution and \ts ``$\Rightarrow$'' \ts denotes the weak convergence.
	\end{conjecture}
	
	Note that $B_{c}C + (B_{c}C)^{2}$ is the variance of the geometric distribution with mean $B_{c}C$.
	We remark that currently we only know that, for all \ts $\alpha<\delta/3$ and $B>B_{c}$, we have:
	$$\E\bigl[X_{1,n+1}^{2}\land n^{\alpha}\bigr] \. \longrightarrow \. (B_{c}C)^{2} \quad \text{as} \ \ n\rightarrow \infty.
	$$
	Unfortunately, we are not able to get rid of the truncation in this formula, in contrast to Lemma 3.4 in \cite{CDS} for the uniform margin case. This is partly because our argument relies on the loose estimates of $|\mathcal{M}(\r,\c)|$ given by Theorem \ref{thm:Barvinok_summary}. Replacing the LHS with \ts $\E[X_{1,n+1}^{2}]$ \ts
	would be the first step in proving Conjecture~\ref{conj:CLT}.
	
	Next, we conjecture that there exists a phase transition in $\delta$ with respect to the limiting distribution of $X_{11}$ in the supercritical regime $B>B_{c}$. For the \emph{thick bezel case} \ts $1/2<\delta<1$, Theorem~\ref{thm:main_geo} shows that $n^{\delta-1}X_{11}$ converges in distribution to a geometric random variable. For the \emph{thin bezel case} \ts $0<\delta<1/2$, we conjecture that it should converge to a normal distribution. Roughly speaking, the sum of $n^{\delta}$ terms $X_{11}+\ldots+X_{1,\lfloor n^{\delta} \rfloor}$ is asymptotically a normal random variable by Conjecture~\ref{conj:CLT}. Thus, if $\delta<1/2$, then there are not enough terms in this sum to exhibit central limit behavior. Hence the limiting distribution of this sum should be some rescaled version of the marginal distribution of $X_{11}$.
	
	To make a more precise conjecture, let $S_{n,\delta}(B,C)$ be as in Conjecture~\ref{conj:Gaussian}. Then we write:
	\begin{align}
	\frac{1}{\sqrt{n}} \. \sum_{k=1}^{\lfloor n^{\delta} \rfloor} \.
	\bigl[X_{1,k} - C(B-B_{c})\ts n^{1-\delta}\bigr] \, = \, \frac{1}{\sqrt{n}} \. \Bigl(S_{n,\delta}(B,C) - B_{c}Cn\Bigr).
	\end{align}
	Assuming the summands in the left hand side are asymptotically uncorrelated, taking variance in each side gives
	\begin{align}
	\text{Var}(X_{11}) \. \sim \. n^{1-\delta}\bigl(B_{c}C + (B_{c}C)^{2}\bigr).
	\end{align}
	Hence we have at the following conjecture.
	
	\begin{conjecture}\label{conj:Gaussian}
		Fix $B,C>0$ and $0< \delta<1/2$. Let $X=(X_{ij})$ be sampled from $\M_{n,\delta}(B,C)$ uniformly at random.
		Then
		\begin{align}
		\frac{X_{11} - C(B-B_{c})\ts n^{1-\delta}}{n^{(1-\delta)/2}\sqrt{B_{c}C + (B_{c}C)^{2}}} \, \Longrightarrow \, \nN(0,1)\ts.
		\end{align}
	\end{conjecture}
	
	Further conjectures and open problems are given in Section~\ref{s:finrem}.

	\vspace{0.2cm}
	
	\section{Concentration in blocks}

	As we mentioned in the introduction, Barvinok~\cite{B3} introduced the notion of \textit{typical table} for general contingency tables, which captures some ``typical behavior'' of the uniform random contingency table of fixed margins.  We start with the precise definition:

	\begin{definition}[Typical table]\label{def:typical_table}
		Fix margins $\r\in \nn^{m}$ and $\c\in \nn^{n}$.
		Let $\mathcal{P}(\r,\c)\subseteq \R_+^{mn}$ denote the transportation polytope.
		For each $X=(X_{ij})\in \mathcal{P}(\r,\c)$, define
		\begin{align}\label{eq:def_g}
		g(X) \. = \, \sum_{1\le i,j\le n} \. f(X_{ij})\ts,
		\end{align}
		where the function $f:[0,\infty)\rightarrow [0,\infty)$ is defined by
		\begin{align}\label{eq:def_f}
		f(x) \. = \. (x+1)\log(x+1) - x\log x\ts.
		\end{align}
		The \textit{typical table} $Z\in \mathcal{P}(\r,\c)$ for $\M(\r,\c)$ is defined by
		\begin{align}
		Z \. = \. \argmax_{X\in \mathcal{P}(\r,\c)} \. g(X).
		\end{align}
	\end{definition}
	
	Here $\mathcal{P}(\r, \c)$ is is the transportation polytope of margins $\r$ and $\c$ defined in the introduction. Since the function $g$ defined at~\eqref{eq:def_g} is strictly concave, it attains a unique maximizer on the transportation polytope $\mathcal{P}(\r,\c)$ and thus the typical table is well-defined.

	One of the building blocks of our main result is Lemma~\ref{lemma:key}, which says that the law of an entry in a large block of random contingency table is attracted toward a geometric distribution, whose mean is dictated by the corresponding typical table. Given the set $\mathcal{M}(\r,\c)$ of \ts $m\times n$ contingency tables of margins $\r$ and $\c$, we call a set $\mathcal{B}\subseteq \{1,\ldots,m\}\times \{1,\ldots,n\}$ of indices a \textit{block} of \ts $\mathcal{M}(\r,\c)$ \ts if
	\begin{align}
	(r_i, c_j) \. = \. (r_{i'}, c_{j'}) \quad \text{for all \, $(i,j), \. (i',j')\ts \in \mathcal{B}$}.
	\end{align}
	Observe that when $X=(X_{ij})$ is sampled from $\mathcal{M}(\r,\c)$ uniformly at random
	and $\mathcal{B}$ is a block of $\mathcal{M}(\r,\c)$, the entries $X_{ij}$ for all
	$(i,j)\in \mathcal{B}$ have the same distribution by the symmetry.
	Moreover, the entries of the typical table $Z$ within a block are the same.

	\begin{lemma}\label{lemma:key}
		Let $\mathcal{M}(\r,\c)$ be the set of all \ts $m\times n$ contingency tables of margins $\r$ and~$\c$. Let $X=(X_{ij})$ be sampled from $\mathcal{M}(\r,\c)$ uniformly at random. Suppose $\mathcal{B}_{1},\ldots,\mathcal{B}_{k}$ are (not necessarily distinct) blocks in $\mathcal{M}(\r,\c)$ with $|\mathcal{B}_{1}|\le \ldots \le |\mathcal{B}_{k}|$. Then there exists an absolute constant $\gamma>0$, s.t.\ for each $I=(I_1,\ldots,I_k)\in \mathcal{B}_{1}\times \ldots \times \mathcal{B}_{k}$ and $t>0$, we have:
		\begin{align}
		d_{TV}\left( \prod_{\ell=1}^{k} \. X_{I_\ell}, \prod_{\ell=1}^{k} Y_{I_\ell}\right) \le t + N^{\gamma(m+n)} \exp\left( - \left\lfloor \frac{|\mathcal{B}_{1}|}{k} \right\rfloor \frac{t^{2}}{2} \right),
		\end{align}
		where $Z=(z_{ij})$ is the typical table for \ts $\mathcal{M}(\r, \c )$, and $Y=(Y_{J})$ is the random matrix of independent entries with \ts
		$Y_{ij}\sim \Geom(z_{ij})$, and \ts $N=r_1+\ldots+r_m = c_1+\ldots+c_n$.
	\end{lemma}

	Our proof of Lemma~\ref{lemma:key} relies upon the following results of Barvinok, see \cite[Thm~1.7]{B3}, \cite[Thm~1.1]{B1},
	and~\cite[Lem~1.4]{B1}.
	
	\begin{theorem}[\cite{B1, B3}]\label{thm:Barvinok_summary}
		Fix margins $\r\in \nn^{m}$ and $\c\in \nn^{n}$. Let $Z=(z_{ij})$ be the typical table for $\M(\r,\c)$, and let $Y=(Y_{ij})$ be the $(m\times n)$ random matrix of independent entries where \ts $Y_{ij}\sim \Geom(z_{ij})$. Let \. $N := \ts \sum_{i=1}^{m} r_i \ts = \ts \sum_{j=1}^{n} c_j$ \. denote the total sum.
		\begin{description}
			\item[(i)] There exists some absolute constant $\gamma>0$ such that
			\begin{align}
			N^{-\gamma(m+n)} e^{g(Z)} \. \le \. |\M(\r,\c)| \. \le \. e^{g(Z)}\ts.
			\end{align}
			\item[(ii)] Conditioned on being in $\M(\r,\c)$, table $Y$ is uniform on $\M(\r,\c)$\ts.
			\item[(iii)] For the constant $\gamma>0$ in (i), we have
			\begin{align}
			\P\bigl(Y\in \M(\r,\c)\bigr) \. = \. e^{-g(Z)} |\M(\r,\c)| \. \ge \. N^{-\gamma(m+n)}\ts.
			\end{align}
		\end{description}
	\end{theorem}
	
	In other words, (ii) and (iii) of the above theorem says that the \emph{geometric matrix}~$Y$ with mean given by the typical table $Z$ emulates the uniform random table $X$ in $\M(\r,\c)$ with probability at least $N^{-\gamma(m+n)}$. Hence on very rare events,
	we can `transfer' some of the properties of this geometric matrix~$Y$ to the uniform random contingency table~$X$.

	Now we are ready to prove the key lemma.
	
	\begin{proof}[\textbf{Proof of Lemma~\ref{lemma:key}}]
		Let $\mathcal{B}_{1},\ldots, \mathcal{B}_{k}$ \ts be (not necessarily distinct) blocks in \ts $\M(\r, \c)$  with $|\mathcal{B}_{1}|\le \ldots \le |\mathcal{B}_{k}|$. Let $Z=(z_{ij})$ be the typical table for \ts $\mathcal{M}(\r,\c)$, and let \ts $Y=(Y_{J})$ \ts denote the random matrix of independent entries where $Y_{ij}\sim \Geom(z_{ij})$. Observe that we can choose a subset $\mathcal{I}\subseteq \mathcal{B}_{1}\times \ldots \times \mathcal{B}_{k}$ such that $|\mathcal{I}|\ge \lfloor |\mathcal{B}_{1}|/k \rfloor$ and every matrix coordinate $(i,j)$ appears in at most one element of $\mathcal{I}$. Indeed, in the worst case when $\mathcal{B}_{1}=\ldots=\mathcal{B}_{k}$, we may partition $\mathcal{B}_{1}$ into sub-blocks so that we have at least $\lfloor |\mathcal{B}_{1}|/k \rfloor$ sub-blocks of size $k$. We can then turn each sub-block into an element of $\mathcal{I}$. A similar argument holds for the general case.
		
		Fix measurable sets \ts $A\subseteq [0,\infty)$ \ts and \ts
		$\mathcal{A}\subseteq \R_+^{m n}$. For a \ts $m\times n$ matrix $W=(w_{ij})$ \ts and \ts $I\in \mathcal{B}_{1}\times \ldots \times \mathcal{B}_{k}$, denote
		\begin{align}
		W^{I} \, := \, \prod_{\ell=1}^{k} \. w_{I_\ell}, \qquad S(W) \, := \, \frac{1}{|\mathcal{I}|}\. \sum_{I\in \mathcal{I}} \, \mathbf{1}_{\{W^{I} \in A\}}\,.
		\end{align}
		Note that by the exchangeability of the entries of $X$ and $Y$ in each block of \ts $\mathcal{M}(\r,\c)$, variables \ts
		$X^{I}$ \ts and \ts $Y^{I}$ \ts have the same distribution for all~$I$. In particular, we have:
		\begin{align}
		\P(X^{I}\in A) \. = \. \E\bigl[S(X)\bigr]\ts.
		\end{align}
		Moreover, since $Y_{ij}$ are independent and since every two elements of \ts $\mathcal{I}$ \ts
		have non-overlapping coordinates, it follows that \ts $Y^{I}$ \ts are also independent.

		Now note that from Theorem~\ref{thm:Barvinok_summary}~(ii) and~(iii), we have:
		\begin{align}\label{eq:pf_thm1_1}
		\P(Y\in \mathcal{A}) \, \ge \, \P\bigl(Y\in \mathcal{A}\,|\, Y\in \M(\r,\c)\bigr) \ts N^{-\gamma (m+n)}
		\, = \, \P(X\in \mathcal{A}) \ts N^{-\gamma (m+n)}.
		\end{align}
		Also, by the Azuma--Hoeffding inequality, for every fixed $I\in \mathcal{I}$, we have:
		\begin{align}
		\P\Bigl(  \left| S(Y) - \P(Y^{I}\in A) \right|  \. \ge \. t  \Bigr) \, \le \,
		2\ts \exp\left( - \frac{|\mathcal{I}|\ts t^{2}}{2} \right) \, \le \, 2\ts \exp\left( - \frac{|\mathcal{B}_{1}|\ts t^{2}}{2k} \right).
		\end{align}	
		Hence, by conditioning on whether \ts $\left| S(X) -  \P(Y_{I}\in A) \right|$ \ts is small or large, we get
		\begin{align}
		| \P(X^{I}\in A ) -  \P(Y^{I}\in A ) | & \, = \, \Bigl| \E\bigl[S(X)\bigr] -  \P(Y^{I}\in A ) \Bigr| \, \le \,
		\E\Bigl[ \left| S(X) -  \P(Y^{I}\in A) \right| \Bigr]  \\
		& \, \le \, t \ts \P\left( \left| S(X) -  \P(Y^{I}\in A) \right| \le t \right)  +  2\ts \P\left( \left| S(X) -  \P(Y^{I}\in A) \right| > t \right) \\
		&\, \le \, t + 4\ts N^{\gamma (m+n)} \exp\left( - \frac{|\mathcal{B}_{1}|\ts t^{2}}{2 k} \right).
		\end{align}
		Since $A\subseteq [0,\infty)$ is arbitrary, by absorbing the factor of~4 into~$\gamma$, we obtain the result.
	\end{proof}	
	
	\begin{remark}\label{r:high-dim}
		Following the arguments in~\cite{B1, B3}, it is not hard to see that a higher dimensional analogue of Theorem~\ref{thm:Barvinok_summary} holds. Namely, replace $\M(\r,\c)$ with $\M(\r_{1},\ldots,\r_{d})$ and $m+n$ with $n_{1}+\ldots+n_{d}$ in the theorem. Of course, the constant $\gamma=\gamma(d)>0$ then depends on~$d$. Then a similar argument will show that a higher dimensional analogue of Lemma \ref{lemma:key} also holds. Hence most of our main results should hold in higher dimensions. We do not justify this claim in the present paper.
	\end{remark}

	\vspace{0.2cm}

	\section{Phase transition in the typical table and rate of convergence}
	
	Recall the definition of the typical table $Z$ for $\M(\r,\c)$ given in the introduction. Namely, $Z$ is the unique maximizer of the function $g$ defined at $\eqref{eq:def_g}$ on the transportation polytope $\mathcal{P}(\r,\c)$. Note that $\mathcal{P}(\r,\c)$ is defined by the intersection of the hyperplanes in $\mathbb{R}^{mn}_+$ given by
	\begin{align}
	h_{i\bullet}\. := \. \left(\sum_{k=1}^{m} \. x_{ik}\right) \. - \. r_i \. = \. 0 \qquad \text{for all} \ \ 1\le i \le m\ts, \\
	h_{\bullet j}\. := \. \left(\sum_{k=1}^{n} \. x_{kj}\right) \. - \. c_j \. = \. 0 \qquad \text{for all} \ \  1\le j \le n\ts.
	\end{align}
	Also, given that the margins $\r=(r_{1},\ldots, r_{m})$ and $\c=(c_{1},\ldots, c_{m})$ have strictly positive coordinates (which is true in our setting), the transportation polytope $\mathcal{P}(\r,\c)$ contains a matrix $W=(w_{ij})$ whose entries are all strictly positive. For instance, the expectation under the Fisher-Yates distribution \eqref{eq:Fisher_Yates} has the form $W=(w_{ij})$ with $w_{ij}=r_{i}c_{j}/N>0$, where $N=r_{1}+\ldots+r_{m}$. It follows that the relative interior of $\mathcal{P}(\r,\c)$ contains an $\eps$-ball around~$W$ for some $\eps>0$. This allows us to use Lagrange's method to maximize~$g$ over the polytope $\mathcal{P}(\r,\c)$.

	Note that the gradient $\nabla h_{i\bullet}$ is the \ts $m\times n$ matrix $E_{i\bullet}$, which has 1's in the $i$-th row and 0's elsewhere. Similarly,
	$\nabla h_{j\bullet}$ is the \ts $m\times n$ matrix  $E_{\bullet j}$,
	which has 1's in the $j$-th column and 0's elsewhere. On the other hand, it is easy to see that the gradient $\nabla g$ of the objective function $g$ defined at $\eqref{eq:def_g}$ is given by
	\begin{align}
	\nabla g \. = \. \bigl( \log\left(1 \ts+\ts 1/x_{ij}\bigr) \right)_{1\le i \le m, 1\le j \le n}\,.
	\end{align}
	Hence by the multivariate Lagrange's method, when evaluated at the typical table $Z=(z_{ij})$, $\nabla g$ must be in the non-negative span of $\nabla h_{i\bullet}$'s and $\nabla h_{\bullet j}$'s. This gives that there exists some non-negative constants $\lambda_{1},\ldots,\lambda_{m}$ and $\mu_{1},\ldots,\mu_{n}$ such that
	\begin{align}
	\log\left(1+ x^{-1}_{ij}\right) \.  = \. \lambda_{i} + \mu_{j} \qquad \text{for all} \ \  1\le i \le m,\, 1\le j \le n,
	\end{align}
	or equivalently,
	\begin{align}\label{eq:formula_entry_Z}
	z_{ij} \. = \. \frac{1}{\exp(\lambda_{i}+\mu_{j})-1}  \qquad \text{for all} \ \  1\le i \le m,\, 1\le j \le n.
	\end{align}

	Now we consider $\mathcal{M}(\r,\c)=\mathcal{M}_{n,\delta}(B,C)$, where the margins $\r$ and $\c$ are given at~\eqref{eq:def_barvinok_margin}. By symmetry, there exist some constants $\alpha,\beta>0$, possibly depending on all parameters, such that
	\begin{align}
	\log\left(1+ z^{-1}_{ij}\right) =
	\begin{cases}
	2\alpha & \text{if $1\le i,j\le \lfloor n^{\delta} \rfloor\., $} \\
	2\beta & \text{if $\lfloor n^{\delta} \rfloor< i,j\le \lfloor n^{\delta} \rfloor+ n\.,$} \\
	\alpha+\beta & \text{otherwise}\ts.
	\end{cases}
	\end{align}
	Furthermore, denote $P=P_{n}=\exp(\alpha)$ and $Q=Q_{n}=\exp(\beta)$. Then~\eqref{eq:formula_entry_Z} gives
	\begin{align}\label{eq:def_P_Q}
	z_{11} = \frac{1}{P^{2}-1},\qquad z_{1,n+1} = \frac{1}{PQ-1}, \qquad z_{n+1,n+1} = \frac{1}{Q^{2}-1}.
	\end{align}

	Note that the margin condition for $Z$ reduces to
	\begin{align}\label{eq:z_ij_reduced}
	\begin{cases}
	(\lfloor n^{\delta} \rfloor/n) \ts z_{11} + z_{1,n+1} \. = \. BC\ts, \\
	(\lfloor n^{\delta} \rfloor/n) \ts z_{1,n+1} + z_{n+1,n+1} \. = \. C\ts.
	\end{cases}
	\end{align}
	For a preliminary analysis for the solution of the equations in~\eqref{eq:z_ij_reduced}, observe that
	\begin{align}\label{eq:limit_z_n+1,n+1}
	z_{n+1,n+1}\le C, \quad z_{1,n+1}\le BC, \quad |z_{n+1,n+1} - C| = n^{\delta-1}z_{1,n+1} \le BC n^{\delta-1}.
	\end{align}
	In particular, $z_{n+1,n+1}\rightarrow C$ as $n\rightarrow \infty$.

	The main result in this section is the following lemma, which establishes the phase transition of the typical table and the rate of convergence of its entries.
	
	\begin{lemma}\label{lemma:typical_corner}
		Let $Z=(z_{ij})$ be the typical table for $\M_{n,\delta}(B,C)$, where $0\le \delta<1$. Let $B_{c}=1+\sqrt{1+1/C}$ be as above.
		\begin{description}
			\item[(i)] If $B<B_{c}$, then
			\begin{align}
			z_{11} = \frac{B^{2}C(C+1)}{ (B_{c}-B)(B_{c}+B-2) } + O(n^{\delta-1}), \qquad  z_{1,n+1}  = BC + O(n^{\delta-1}).
			\end{align}
			
			\item[(ii)] If $B>B_{c}$, then
			\begin{align}
			z_{n+1,n+1}=C + O(n^{\delta-1}), \qquad  z_{1,n+1} = B_{c}C + O(n^{\delta-1}), \qquad n^{\delta-1}z_{11} =  C(B-B_{c}) + O(n^{\delta-1}).
			\end{align}
		\end{description}
		where the constants in $O(n^{\delta-1})$ remain bounded as $B\rightarrow \infty$.
	\end{lemma}

	In the following proposition, we show that if $B<B_{c}$, then the corner entry $z_{11}$
	of the typical table for $\M_{n,\delta}(B,C)$ is uniformly bounded in~$n$.
	We remark that there is a more general result of this type. Namely, \cite[Thm~3.5]{BLSY}
	states that if the row and columns do not vary much, then the entries of the typical table
	are uniformly bounded by some constant independent of the size of the table.
	However, this result gives a sub-optimal lower critical value \ts $B<(1+\sqrt{1+4(1+1/C)})/2$.
	In order to push this threshold up to the desired critical value \ts $B<B_{c}$\ts,
	we optimize the proof of \cite[Thm~3.5]{BLSY} for our model.
	
	\begin{prop}\label{prop:z_11_subcritical}
		In notation above, suppose \ts $B<B_{c}$. Then:
		\begin{align}
		\limsup_{n\rightarrow \infty} \ts z_{11} \, \le \, \frac{B^{2}C(C+1)}{ (B_{c}-B)(B_{c}+B-2) } \, < \. \infty\ts.
		\end{align}
	\end{prop}
	
	\begin{proof}
		For brevity, denote $w_{11}=z_{11}$, $w_{21}=w_{12}=z_{1,n+1}=z_{n+1,1}$, and $w_{22}=z_{n+1,n+1}$. Then
		\begin{align}
		w_{11} \ts = \. \frac{1}{P^{2}-1}\,, \quad w_{21}\ts = \ts w_{22} \ts = \. \frac{1}{PQ-1}\,,
		\quad w_{22} \ts = \.  \frac{1}{Q^{2}-1}\,.
		\end{align}
		Note that
		\begin{align}
		w_{11}\ts w_{22} \. = \. \frac{1}{(P^{2}-1)(Q^{2}-1)} \. \ge  \. \frac{1}{(PQ-1)^{2}} \. = \. w_{12}\ts w_{21}\..
		\end{align}
		Let us show that
		\begin{align}
		w_{21}/w_{22} \. \le \. B.
		\end{align}
		Assume otherwise, that $w_{21}/w_{22} > B$.  The above inequality gives \ts $w_{11}/w_{12}\ge w_{12}/w_{22} > B$, and we get
		\begin{align}
		BC & \. =  \. \lfloor n^{\delta} \rfloor w_{11}/n + w_{12} \, > \, \lfloor n^{\delta} \rfloor \ts Bw_{21}/n + Bw_{22} \. = \. BC\ts,
		\end{align}
		a contradiction.
		
		By the definition of $P$ and $Q$, we have:
		\begin{align}
		\frac{(w_{22}+1)w_{12}}{(w_{12}+1)w_{22}} \, = \. Q^{2}/PQ \. = \. Q/P.
		\end{align}
		In order to upper bound $Q/P$, we consider maximizing the fraction in the left hand side. By~\eqref{eq:limit_z_n+1,n+1}, we know that \ts $\lim_{n\rightarrow \infty} w_{22}= C$. Hence we have the following optimization problem:
		\begin{align}
		&\text{maximize} \ \quad \frac{(w_{22}+1)w_{12}}{(w_{12}+1)w_{22}} \\
		& \text{subject to} \quad \ w_{22}\le w_{12} \le B w_{22} \quad \text{and} \quad w_{22}=C+o(1).
		\end{align}
		It is not hard to see that the objective function is non-decreasing in $w_{12}$ and non-increasing in $w_{22}$. Hence, as $n\rightarrow \infty$, the solution to the above problem approaches the limit $B(C+1)/(BC+1)$. This implies
		\begin{align}
		\limsup_{n\rightarrow \infty} \. \frac{Q}{P} \, \le \, \frac{B(C+1)}{BC+1}\..
		\end{align}
		Now, since \ts $w_{12}\le BC$ \ts by~\eqref{eq:limit_z_n+1,n+1}, we have:
		\begin{align}
		\liminf_{n\rightarrow \infty} \. P^{2} & \, = \,\liminf_{n\rightarrow \infty} \. PQ / (Q/P)
		\,\ge \, \liminf_{n\rightarrow \infty} \, \frac{(w_{12}+1)\cdot P}{w_{12}\cdot Q} \\
		& \,\ge\, \frac{(BC+1)\cdot (BC+1)}{BC\cdot B(C+1)}\, = \,\frac{(BC+1)^{2}}{B^{2}C(C+1)}\..
		\end{align}
		This implies
		\begin{align}
		\limsup_{n\rightarrow \infty} \. w_{11} \, = \, \limsup_{n\rightarrow \infty} \. \frac{1}{P^{2}-1}
		\, \le \, \frac{B^{2}C(C+1)}{ 2B - B^{2} + C^{-1} } \, = \, \frac{B^{2}C(C+1)}{ (B_{c}-B)(B_{c}+B-2) }\..
		\end{align}
		This finishes the proof.
	\end{proof}


	\begin{proof}[\textbf{Proof of Lemma~\ref{lemma:typical_corner}}]
		Suppose $B<B_{c}$. By Proposition~\ref{prop:z_11_subcritical}, we know that $z_{11}$ is uniformly bounded in~$n$.
		Hence, from the first equation in~\eqref{eq:z_ij_reduced}, we obtain:
		\begin{align}\label{eq:z_ij_limit_1}
		|z_{1,n+1} - BC| \le n^{\delta-1} z_{11} \.= \. O\bigl(n^{\delta-1}\bigr).
		\end{align}
		In particular, $\lim_{n\rightarrow \infty} z_{1,n+1} = BC$.
		
		For $z_{11}$, let $P=P_{n}$ and $Q=Q_{n}$ be as in~\eqref{eq:def_P_Q}. Recall that $z_{1,n+1} = 1/(PQ-1)$ and $z_{n+1,n+1} = 1/(Q^{2}-1)$. Also recall that $z_{n+1,n+1}\rightarrow C$ as $n\rightarrow \infty$ from~\eqref{eq:limit_z_n+1,n+1}.  Hence, we have:
		\begin{align}
		\lim_{n\rightarrow \infty}\frac{1}{PQ-1} = BC, \qquad \lim_{n\rightarrow \infty} \frac{1}{Q^{2}-1} = C.
		\end{align}
		This implies \ts $Q\rightarrow q_{*}:=\sqrt{1+1/C}$ \ts and \ts
		$P\rightarrow p_{*}:=(1+1/BC)/\sqrt{1+1/C}>1$ \ts as \ts $n\rightarrow \infty$.
		It follows that \ts $z_{11}\rightarrow 1/(p_{*}^{2}-1)$ \ts as \ts $n\rightarrow \infty$,
		which is the correct limit that~(i) implies.
		
		In order to obtain the rate of convergence of $z_{11}$, first define a
		function \ts $h(x) = 1/(x^{2}-1)$. Then, since \ts $z_{11} = 1/(P^{2}-1)$,
		it suffices to show
		\begin{align}\label{eq:z_ij_limit_2}
		|h(P) - h(p^{*})| \. = \. O\bigl(n^{\delta-1}\bigr).
		\end{align}
		For this end, first note that \ts $h(x)'=-2x/(x^{2}-1)^{2}$ \ts and
		\ts $|h'(x)|$ \ts is a decreasing function in \ts $(1,\infty)$. Hence by the mean value theorem,
		for every constant \ts $1<p<p^{*}$, we have:
		\begin{align}\label{eq:z_ij_limit_3}
		|h(P) - h(p^{*})| \,\le\, |h'(p)| \cdot |P-p^{*}|
		\end{align}
		for all sufficiently large \ts $n\ge 1$.
		
		Next, write
		\begin{align}\label{eq:z_ij_limit_4}
		|P-p^{*}| &\, \le \, \left|  P- \frac{1+1/BC}{Q} \right| \. + \. (1+1/BC) \left| \frac{1}{Q} - \frac{1}{q^{*}} \right|.
		\end{align}
		Since \ts $z_{n+1,n+1} = 1/(Q^{2}-1) = h(Q)$ and $C=h(q_{*})$, the mean value theorem implies that
		\begin{align}\label{eq:z_ij_limit_5}
		|z_{n+1,n+1}-C| \, = \, |h(Q) - h(Q_{*})| \, \ge\, |h'(2q_{*})| \cdot |Q-q_{*}|
		\end{align}
		for all sufficiently large $n\ge 1$. Thus~\eqref{eq:limit_z_n+1,n+1} gives $|Q-q_{*}|= O(n^{\delta-1})$. Since $Q\rightarrow q_{*}$ as $n\rightarrow \infty$, this also implies that the second term in~\eqref{eq:z_ij_limit_4} is of order $O(n^{\delta-1})$. For the first term, note that
		\begin{align}
		\left|  P- \frac{1+1/BC}{Q} \right| \,  = \, \frac{(PQ-1)/BC}{Q} \,
		\left| \frac{1}{PQ-1} - BC \right|\ts.
		\end{align}
		Since $z_{1,n+1}=1/(PQ-1)$ and both $P$ and $Q$ converge as $n\rightarrow\infty$, the above expression is \ts $O(n^{\delta-1})$. Thus $|P-p_{*}|=O(n^{\delta-1})$, and~\eqref{eq:z_ij_limit_2} follows from~\eqref{eq:z_ij_limit_3}. This shows (i).

		Next, suppose $B>B_{c}$. To show (ii), we first obtain a lower bound on~$z_{11}$.
		Note that~\eqref{eq:limit_z_n+1,n+1} gives $1/(Q^{2}-1) = z_{n+1,n+1} \le C$, so we have
		\begin{align}
		Q \, \ge \, \sqrt{1+C^{-1}} \ts. 
		\end{align}
		Then from the first equation in~\eqref{eq:z_ij_reduced} and the fact that $P\ge 1$, we have
		\begin{align}
		z_{11} &\, = \, (n/\lfloor n^{\delta}\rfloor) (BC-z_{1,n+1}) \, \ge\, n^{1-\delta}\left(BC-\frac{1}{Q-1}\right) \\
		&\, \ge\, n^{1-\delta} \left(BC-\frac{1}{\sqrt{1+1/C}-1}\right) \, = \, n^{1-\delta} C(B-B_{c})\ts.  \label{eq:z_11_lowerbd}
		\end{align}
		
		Now we derive the limits in~(ii). First, note that from~\eqref{eq:z_11_lowerbd}, we have
		\begin{align}
		\liminf_{n\rightarrow \infty} \. \frac{n^{\delta-1}}{P^{2}-1} \, =\,
		\liminf_{n\rightarrow \infty} \. n^{\delta-1}z_{11} \, \ge \, C(B-B_{c})\, >\. 0\ts.
		\end{align}
		Since $\delta<1$, we must have $\lim_{n\rightarrow \infty} P_{n} = 1$. Moreover, $\lim_{n\rightarrow \infty} Q_{n} = q_{*}=\sqrt{1+1/C}$ by~\eqref{eq:limit_z_n+1,n+1}. We conclude that \ts $\lim_{n\rightarrow \infty} z_{1,n+1} = 1/(q_{*}-1) = B_{c}C$.
		
		Finally, we derive the rate of convergence. First, using~\eqref{eq:limit_z_n+1,n+1} and the limit of $z_{1,n+1}$ we have just shown, we have
		\begin{align}\label{eq:z_ij_limit_7}
		|z_{n+1,n+1}-C| \,\le\, 2B_{c}C\ts n^{\delta-1}
		\end{align}
		for all sufficiently large $n\ge 1$. Also, from~\eqref{eq:z_11_lowerbd}, we have
		\begin{align}\label{eq:z_ij_limit_6}
		0\. \le \. P-1 \. \le \. \frac{n^{\delta-1}}{C(B-B_{c})} \quad \text{for all} \ \. n\ge 1.
		\end{align}

		Using the notation $q_{*}=\sqrt{1+1/C}$ as above, we can write
		\begin{align}
		|z_{1,n+1} - B_{c}C|\, =\, \left| z_{1,n+1} - \frac{1}{q_{*}-1} \right|
		\, \le \, \frac{Q(P-1)}{(PQ-1)(Q-1)} \. + \.\left|  \frac{1}{Q-1} - \frac{1}{q_{*}-1}  \right|.
		\end{align}
		By~\eqref{eq:z_ij_limit_6}, the first term in the right hand side is \ts $O(n^{\delta-1}/(B-B_{c}))$, where the constant does not depend on $B$. On the other hand, note that the inequality~\eqref{eq:z_ij_limit_5} still holds for $B>B_{c}$. Hence~\eqref{eq:z_ij_limit_7} implies that the second term is of order $O(n^{\delta-1})$, where the constant is independent in $B$. Lastly, it then follows that
		\begin{align}
		\bigl|n^{\delta-1}z_{11} - C(B-B_{c})\bigr| \, = \, |B_{c}C - z_{1,n+1}| \, = \, O\bigl(n^{\delta-1}\bigr),
		\end{align}
		where the constant in the last $O(n^{\delta-1})$ does not grow in $B$. This proves~(ii).
	\end{proof}

	\vspace{0.2cm}
	\section{Convergence in total variation distance and the proof of Theorem~\ref{thm:main_geo}}

	In this section we prove Theorem~\ref{thm:main_geo}. We first derive the following result from Lemma~\ref{lemma:key}, which states that each entry of the random contingency table $X\in \mathcal{M}_{n,\delta}(B,C)$ has a geometric limiting distribution with mean dictated by the corresponding entry in the typical table.
	
	\begin{theorem}\label{thm:geo_approximation}
		Let $X=(X_{ij})$ be sampled from $\M_{n,\delta}(B,C)$ uniformly at random. Let $Z=(z_{ij})$ be the typical table for $\mathcal{M}_{n,\delta}(B,C)$ and let $Y=(Y_{ij})$ denote the random matrix of independent entries with $Y_{ij}\sim \Geom(z_{ij})$. Fix integers $k\ge 1$ and $1\le i_{r},j_{r}\le n+\lfloor n^{\delta} \rfloor$, $1\le r\le k$. Then for every $\eps>0$, we have:
		\begin{align}
		d_{TV}\left( \prod_{r=1}^{k} X_{i_{r},j_{r}}, \prod_{r=1}^{k} Y_{i_{r},j_{r}}\right)
		\le n^{-\eta(\delta) + \eps}
		\end{align}
		for all sufficiently large $n\ge 1$, where
		\begin{align}\label{eq:def_eta_exponent}
		\eta(\delta) =
		\begin{cases}
		1/2 & \text{if all $(i_{r},j_{r})$'s are contained in the bottom right block,} \\
		\delta-1/2  & \text{if some $(i_{r},j_{r})$ is contained in the top left block,} \\
		\delta/2 & \text{otherwise}.
		\end{cases}
		\end{align}
	\end{theorem}

	\begin{proof}
		Fix $\eps>0$. Note that each $(i_{r},j_{r})$ is contained one of the four blocks of $\mathcal{M}_{n,\delta}(B,C)$. Suppose that all of the indices are contained in the bottom right block \ts $\mathcal{B}_{BR}:=\bigl\{ \lfloor n^{\delta} \rfloor+1, \lfloor n^{\delta} \rfloor+n  \bigr\}^{2}$, which has size $n^{2}$. Then we apply Lemma~\ref{lemma:key} with $t=(1/2)n^{-1/2+\eps}$. Since $N \le Cn^{2} + BCn^{1+\delta}$ and $0\le \delta<1$, this gives:
		\begin{align}
		d_{TV}\left( \prod_{r=1}^{k} \ts X_{i_{r},j_{r}}, \ts \prod_{r=1}^{k} \ts Y_{i_{r},j_{r}}\right) &\, \le \, \frac12n^{-1/2+\eps} \. + \. \exp(c\ts n\log n) \ts \exp\bigl(-c\ts n^{1+2\eps}\bigr) \\
		&\, \le \, \frac12\. n^{-1/2+\eps}\. + \. \exp\bigl(-c'n^{1+2\eps}\bigr)\ts,
		\end{align}
		for some constants $c,c'>0$. The equation in the theorem follows from here.  Note that the top left block has size $(\lfloor n^{\delta} \rfloor)^{2}$ and the top right and bottom left blocks have size $n\lfloor n^{\delta}\rfloor$. By applying
		Lemma~\ref{lemma:key} for $t=(1/2)n^{-(\delta-1/2) + \eps}$ and $t=(1/2)n^{-(\delta/2)+\eps}$ for the other cases,
		respectively, the theorem now follows from a similar argument.
	\end{proof}
	
	\begin{remark}\label{remark:thm_geo_approx_marginal}
		For marginal distribution of single entry of $X\in \mathcal{M}_{n,\delta}(B,C)$, our Theorem~\ref{thm:geo_approximation} implies:
		\begin{align}
		\begin{cases}
		d_{TV}(X_{n+1,n+1},Y_{n+1,n+1}) \le n^{-1/2+\eps} \\
		d_{TV}(X_{1,n+1},Y_{1,n+1}) \le n^{-(\delta/2)+\eps}\\
		d_{TV}(X_{11},Y_{11}) \le n^{1/2 - \delta+ \eps}\ts,
		\end{cases}
		\end{align}
		for all $\eps>0$ and sufficiently large $n\ge 1$.
	\end{remark}

	Next, we give a simple upper bound on the total variation distance between two geometric distributions in terms of the difference of their mean.

	\begin{prop}\label{prop:TV_geom_0}
		For all \ts $\lambda,\lambda'>0$, we have:
		\begin{align}
		d_{TV}(\Geom(\lambda),\Geom(\lambda'))\,\le\, 2|\lambda-\lambda'| \min\left( \frac{1+\lambda}{1+\lambda'},\, \frac{1+\lambda'}{1+\lambda} \right).
		\end{align}
	\end{prop}	
	
	\begin{proof}
		Suppose $0\le p<q \le 1$. Then for every $k\ge 0$,
		\begin{align}
		&(1-p)p^{k} - (1-q)q^{k} \, = \, (p^{k} - q^{k}) - (p^{k+1}-q^{k+1}) \\
		&\qquad = \, (p-q)\left[ (p^{k-1} + p^{k-2}q + \ldots + pq^{k-1} + q^{k}) -
		(p^{k} + p^{k-1}q + \ldots + pq^{k} + q^{k+1}) \right],
		\end{align}
		so this gives
		\begin{align}
		\sum_{k=0}^{\infty}\. \bigl|(1-p)p^{k} - (1-q)q^{k}\bigr| \, \le \, 2|p-q| \. \sum_{k=0}^{\infty}\. (k+1)q^{k} \,\le\, \frac{2|p-q|}{(1-q)^{2}}.
		\end{align}
		Fix a constant $\lambda>0$ and let \ts $X\sim \Geom(\lambda)$. Then:
		\begin{align}
		\P(X=k) \, = \,\left( \frac{1}{1+\lambda} \right) \left( \frac{\lambda}{1+\lambda}\right)^{k}, \qquad \text{$k\in \nn$}\ts.
		\end{align}
		For $X'\sim \Geom(\lambda')$ with $\lambda'>0$, we have:
		\begin{align}
		2\ts d_{TV}\bigl(\Geom(\lambda),\Geom(\lambda')\bigr) \, = \, \sum_{k=0}^{\infty} \bigl|\P(X=k)-\P(X'=k)\bigr| \, \le \, 2\ts\left|\frac{1}{1+\lambda} - \frac{1}{1+\lambda'}\right| (1+\lambda)^{2} \, = \, 2|\lambda-\lambda'| \frac{1+\lambda}{1+\lambda'}.
		\end{align}
		Then the result follows by changing the roles of $\lambda$ and $\lambda'$. 
	\end{proof}


	\begin{proof}[\textbf{Proof of Theorem~\ref{thm:main_geo}}]
		
		In order to show (i), first note that by Proposition~\ref{prop:TV_geom_0} and~\eqref{eq:limit_z_n+1,n+1}, we get:
		\begin{align}
		d_{TV}\bigl(\Geom(z_{n+1,n+1}),\Geom(C)\bigr) \,\le\, |z_{n+1,n+1}-C|\frac{1+C}{1+z_{n+1}} \,\le\, n^{\delta-1}BC(1+C).
		\end{align}	
		Then the bound on the total deviation distance between $X_{n+1,n+1}$ and $\Geom(C)$ follows from Theorem~\ref{thm:geo_approximation} (see also Remark~\ref{remark:thm_geo_approx_marginal})  and the triangle inequality. This shows (i).
		
		Next, suppose $B<B_{c}$. Then, by Theorem~\ref{thm:geo_approximation}~(ii) (see also Remark~\ref{remark:thm_geo_approx_marginal}) and Proposition~\ref{prop:TV_geom_0}, for all $\eps>0$ and sufficiently large $n\ge 1$, we have:
		\begin{align}
		d_{TV}\bigl(X_{1,n+1}, \Geom(BC)\bigr) &\, \le \, d_{TV}\bigl(X_{1,n+1}, \Geom(z_{1,n+1})\bigr) \. + \. d_{TV}\bigl(z_{1,n+1}, \Geom(BC)\bigr) \\
		&\, \le \, d_{TV}\bigl(X_{1,n+1}, \Geom(z_{1,n+1})\bigr) \. + \. d_{TV}\bigl(z_{1,n+1}, \Geom(BC)\bigr) \\
		&\, \le \, n^{-(\delta/2)+\eps} \. + \. 2|z_{1,n+1} - BC|(1+z_{1,n+1})(1+BC).
		\end{align}
		Therefore, the second term is of order $O(n^{\delta-1})$ by Lemma~\ref{lemma:typical_corner} and \eqref{eq:limit_z_n+1,n+1}. A similar argument can be used to prove the rest of the theorem, except for the case of $X_{11}$ at the supercritical regime.
		
		For the last case, suppose $B>B_{c}$ consider $X_{11}$. Denote $\sigma:=C(B-B_{c})$ and  $\xi_{n}:=n^{1-\delta}\sigma-z_{11}$ for each $n\ge 1$. Note that $\xi_{n}=O(1)$ by Lemma \ref{lemma:typical_corner} (ii). Then for all $\eps>0$ and sufficiently large $n\ge 1$, Theorem~\ref{thm:geo_approximation} implies:
		\begin{align}
		d_{TV}\bigl(X_{11}, \Geom(\sigma n^{1-\delta}+\xi_{n})\bigr) \, & = \, d_{TV}\bigl(X_{11},\Geom(z_{11})\bigr) \le \, n^{1/2-\delta + \eps}.
		\end{align}
		This finishes the proof.
	\end{proof}

	\vspace{0.3cm}
	\section{Convergence of first moments and the proof of Theorem~\ref{thm:main_expectation}}
	
	In this section, we provide convergence of first moments of the entries in the uniform contingency table $X\in \mathcal{M}_{n,\delta}(B,C)$, and prove Theorem~\ref{thm:main_expectation}. Depending on the block that the entries belong to, the convergence may only hold under truncation for general $0< \delta < 1$. For $1/2<\delta<1$, we can get rid of all truncation by using Theorem~\ref{thm:main_geo}.

	We first prove some useful facts that allow us to transfer convergence in total variation distance into $L_{1}$ convergence under a certain uniform boundedness condition. We write $a\land b = \min(a,b)$ for all $a,b\in \mathbb{R}$.

	\begin{prop}\label{prop:geo_truncation}
		Let $(Y_{n})_{n\ge 0}$ be a sequence of random variables taking values from $\nn$ with finite mean. Suppose that there exists some constant $c>0$ such that $\P(Y_{n}\ge k)\le \exp(-ck)$ for all $n,k\ge 0$. Then for every sequence $M_{n}\rightarrow \infty$, we have:
		\begin{align}
		\bigl|\E[Y_{n}\land M_{n}] - \E[Y_{n}]\bigr|\, = \, O\bigl(\exp(-c M_{n})\bigr).
		\end{align}
	\end{prop}
	
	\begin{proof}
		Fix $M>0$. First, observe that
		\begin{align}
		\sum_{k> M} \. \P(Y_{n}\ge k) \, & = \, \sum_{k> M} \sum_{j\ge k} \. \P(Y_{n}=j) \, = \, \sum_{j> M} \sum_{M<k\le j} \.\P(Y_{n}=j) \\
		& = \. \sum_{j> M} \. (j-M) \ts \P(X_{n}=j) \, = \, \E\bigl[Y_{n} - Y_{n}\land M\bigr]\ts.
		\end{align}
		By the exponential tail bound on $Y_{n}$'s this gives
		\begin{align}
		\E[Y_{n} - Y_{n}\land M] \, \le \, \sum_{k>M}  \. e^{-ck} \, =  \, \frac{e^{-c(M+1)}}{1-e^{-c}}.
		\end{align}
		Letting \ts $M:=M_{n}\rightarrow \infty$ as $n\rightarrow \infty$ implies the result.
	\end{proof}

	\begin{prop}\label{prop:TV_expectation}
		Let $(X_{n})_{n\ge 0}$ and $(Y_{n})_{n\ge 0}$ be sequences of random variables taking values from $\nn$ with the following properties:
		\begin{description}
			\item[(i)] There exists some constant $c>0$ such that \ts $\P(Y_{n}\ge k)\le \exp(-ck)$ \ts for all $n,k\ge 0$.
			\item[(ii)] \ts $d_{TV}(X_{n},Y_{n}) = O(n^{-\beta})$ \ts for some $\beta>0$.
		\end{description}
		Then for every sequence $M_{n}\rightarrow \infty$, we have:
		\begin{align}
		\Bigl|\E[X_{n}\land M_{n}] \. - \.\E[Y_{n}] \Bigr|\, = \, O\bigl( M_{n} \ts n^{-\beta}\bigr).
		\end{align}
	\end{prop}

	\begin{proof}
		Fix $M>0$. Write
		\begin{align}
		\E[X_{n}\land M] \. - \. \E[Y_{n}\land M] \,=\, \sum_{k=0}^{M} \.k\ts\delta_{k,n},
		\end{align}
		where \ts $\delta_{k,n} = \P(X_{n}=k) - \P(Y_{n}=k)$. Note that
		\begin{align}
		\sum_{k=0}^{M}|\delta_{k,n}| \,\le\, \sum_{k=0}^{\infty}|\delta_{k,n}| \,=\, 2\ts d_{TV}(X_{n},Y_{n}).
		\end{align}
		Hence we have
		\begin{align}
		\bigl|\E[X_{n}\land M] - \E[Y_{n}\land M] \bigr| \, & \le \left| \sum_{k=0}^{M} k\delta_{k,n} \right| \le M  \sum_{k=0}^{M} |\delta_{k,n}| \le 2M\, d_{TV}(X_{n},Y_{n}).
		\end{align}
		By the assumption (ii), this implies
		\begin{align}
		\bigl|\E[X_{n}\land M_{n}] - \E[Y_{n}\land M_{n}] \bigr| = O(M_{n}n^{-\beta}).
		\end{align}
	 	Now the result follows from the triangle inequality and Proposition~\ref{prop:geo_truncation}.
	\end{proof}

	In the rest of this section, we let $X=(X_{ij})$ denote the random contingency table sampled from
	$\M_{n,\delta}(B,C)$ uniformly at random. Let $Z=(z_{ij})$ be the typical table for $\M_{n,\delta}(B,C)$.
	
	As an immediate application of Lemma~\ref{lemma:typical_corner}, Theorem~\ref{thm:main_geo}, and
	Propositions~\ref{prop:geo_truncation} and~\ref{prop:TV_expectation}, we show that the first moments of the entries in $X$ converge to
	the corresponding entries in $Z$ under truncation.
	
	\begin{prop}\label{prop:first_moment_truncation_bd}
		Let $B,C>0 $ and let \ts $B_{c}=1+\sqrt{1+1/C}$ as above. Fix $\alpha,\eps>0$.
		\begin{description}
			\item[(i)] For all \. $0\le \delta < 1$, we have:
			$$
			\bigl| \E\left[ X_{n+1,n+1}\land n^{\alpha}\right] - z_{n+1,n+1} \bigr| \. =\.  O\bigl(n^{\alpha - 1/2 +\eps}\bigr)\ts.
			$$
			\item[(ii)] For all \. $0< \delta < 1$, we have:
			$$\bigl| \E\left[ X_{1,n+1}\land n^{\alpha}\right] - z_{1,n+1} \bigr| \. = \.O\bigl(n^{\alpha -(\delta/2)+\eps}\bigr)\ts.
			$$
			\item[(iii)] For all \. $1/2<\delta<1$, we have:
			\begin{align}
			\begin{cases}
			\bigl| \E\left[ X_{11}\land n^{\alpha} \right] - z_{11} \bigr| \, = \, O(n^{\alpha +1/2-\delta+\eps}) & \text{if $B<B_{c}$}\., \\
			\left| \E\left[ (n^{\delta-1}X_{11})\land n^{\alpha} \right] - n^{\delta-1}z_{11} \right| \, = \. O(n^{\alpha +1/2-\delta+\eps}) & \text{if $B>B_{c}$}\..
			\end{cases}
			\end{align}
		\end{description}	
	\end{prop}

	\begin{proof}
		We first show (i). Let $Y=(Y_{ij})$ be the matrix of independent entries where $Y_{ij}$
		has geometric distribution with mean $z_{ij}$. Recall that $z_{n+1,n+1}\le C$
		from~\eqref{eq:limit_z_n+1,n+1}. Hence $Y_{n+1,n+1}$ has exponential tail bound
		independent of~$n$. Since $\E[Y_{n+1,n+1}]=z_{n+1,n+1}$, (i) follows from Remark \ref{remark:thm_geo_approx_marginal} and Proposition~\ref{prop:TV_expectation}. A similar argument gives~(ii) and~(iii) for all \ts $B<B_{c}$. Lastly, suppose $B>B_{c}$ and $1/2<\delta<1$. Note that by Lemma~\ref{lemma:typical_corner}, we have:
		\begin{align}
		\lim_{n\rightarrow \infty} n^{\delta-1}z_{11} \,= \, C(B-B_{c}).
		\end{align}
		Thus $n^{\delta-1}Y_{11}$ has exponential tail bound independent of~$n$. Hence (iii) for $B>B_{c}$ also follows from Remark \ref{remark:thm_geo_approx_marginal} and Proposition~\ref{prop:TV_expectation}.
	\end{proof}

	In the following proposition, we show that in many cases, we can get rid of the truncation in Proposition~\ref{prop:first_moment_truncation_bd}. For this purpose, we work with the following `error table' $\ov{X}=(\ov{X}_{ij}) := X - Z$. If we denote $r=\lfloor n^{\delta} \rfloor$, then the margin condition for the tables in $\mathcal{M}_{n,\delta}(B,C)$ give:
	\begin{align}\label{eq:margin_condition_Xbar}
	\begin{cases}
	(\ov{X}_{11} + \ov{X}_{12} + \ldots + \ov{X}_{1r} ) + (\ov{X}_{1,r+1}+\ov{X}_{1,r+2}+\ldots+ \ov{X}_{1,r+n}) &= 0, \\
	(\ov{X}_{n+1,1} + \ov{X}_{n+1, 2} + \ldots + \ov{X}_{n+1, r} ) + (\ov{X}_{n+1,r+1}+\ov{X}_{n+1,r+2}+\ldots+ \ov{X}_{n+1,r+n}) &= 0 .
	\end{cases}
	\end{align}

	\begin{prop}\label{prop:first_moment_error_bound}
		Fix $B,C> 0$ and $0\le \delta<1$. Let $B_{c}=1+\sqrt{1+1/C}$. Then for all $\eps>0$, we have:
		\begin{align}
		\begin{cases}
		\left| \E\left[\ov{X}_{n+1,n+1}  \right] \right| = O(n^{\delta-1}) & \text{$\ \text{for all \ } 0\le \delta< 1$\ts,}\\
		n^{\delta-1} \left| \E[X_{11}] \right| +  \left| \E\left[\ov{X}_{1,n+1}  \right] \right| = O(n^{\delta-1}+n^{-(\delta/2)+\eps}) & \text{\ \text{for all} \  $B<B_{c}$ \, and \, $0<\delta< 1$\ts,} \\
		n^{\delta-1}\left| \E\left[ \ov{X}_{11}  \right]\right| + \left| \E\left[\ov{X}_{1,n+1}  \right] \right| = O(n^{(1/2)-\delta+\eps}) & \text{\ \text{for all} \  $B>B_{c}$ \, and \, $1/2<\delta< 1$}\ts.
		\end{cases}
		\end{align}
	\end{prop}

	\begin{proof}
		Taking expectation for the equations in~\eqref{eq:margin_condition_Xbar} and using symmetry of entries in each block, we obtain:
		\begin{align}\label{eq:prop_1_moement_conv_1}
		\begin{cases}
		(r/n)\ts \E\bigl[\ov{X}_{11}\bigr] + \E\bigl[\ov{X}_{1,n+1}\bigr] \,  = \. 0\ts, \\
		(r/n)\ts \E\bigl[\ov{X}_{n+1,1}\bigr] + \E\bigl[\ov{X}_{n+1,n+1}\bigr] \, = \. 0\ts.
		\end{cases}
		\end{align}
		Note that \ts $\E[X_{n+1,1}]\le BC$ \ts and \ts $z_{n+1,1}\le BC$, so \ts $|\E[\ov{X}_{n+1,1}]|\le 2BC$.
		Hence the second equation in~\eqref{eq:prop_1_moement_conv_1} gives the first equation in the proposition.
		
		In other to show the second and the third equations, suppose $0<\delta<1$. Fix $\alpha,\eps>0$
		and denote $x^{+} = \max(x,0)$. Note that
		\begin{align}\label{eq:prop_1_moement_conv_2}
		\E [ \ov{X}_{1,n+1}] \, = \, \E\left[ X_{1,n+1}\land n^{\alpha} - z_{1,n+1}\right]
		\. + \. \E\left[ \left( X_{1,n+1} - n^{\alpha} \right)^{+} \right].
		\end{align}
		Also, by Proposition~\ref{prop:first_moment_truncation_bd}, we have:
		\begin{align}\label{eq:prop_1_moement_conv_4}
		\bigl|\E\left[ X_{1,n+1}\land n^{\alpha} - z_{1,n+1}\right]\bigr|  \, = \, O\bigl(n^{\alpha- (\delta/2) + \eps}\bigr).
		\end{align}
		
		Suppose $B<B_{c}$. From the first row sum condition for $X$, we have:
		\begin{align}\label{eq:prop_1_moement_conv_3}
		(r/n)\ts \E[X_{11}] +  \E\left[ X_{1,n+1}\land n^{\alpha} - z_{1,n+1}\right] \. + \.
		\E\left[ \left( X_{1,n+1} - n^{\alpha} \right)^{+} \right] \, = \, \frac{\lfloor BC n \rfloor}{n} -z_{1,n+1}\..
		\end{align}
		The right hand side is of order $O({n^{\delta-1}})$ by Lemma~\ref{lemma:typical_corner}.
		Since the first and the last terms in the left hand side are nonnegative, this shows:
		\begin{align}
		n^{\delta-1} \E\bigl[X_{11}\bigr] \. = \. O\bigl( n^{\delta-1}+ n^{\alpha-(\delta/2)+\eps}\bigr),
        \quad \  \E\left[ \left( X_{1,n+1} - n^{\alpha} \right)^{+} \right] \. = \. O\bigl( n^{\delta-1}+ n^{\alpha-(\delta/2)+\eps}\bigr).
		\end{align}
		Furthermore, from~\eqref{eq:prop_1_moement_conv_2} and~\eqref{eq:prop_1_moement_conv_4}, we have:
		\begin{align}
		\left|\E [ \ov{X}_{1,n+1}] \right|  = O( n^{\delta-1}+ n^{\alpha-(\delta/2)+\eps}).
		\end{align}
		Since $\alpha,\eps>0$ were arbitrary, letting $\alpha \to 0$ gives the second equation in the proposition.
		
		It remains to show the last equation the proposition. Suppose $B>B_{c}$ and $1/2<\delta<1$. We rewrite~\eqref{eq:prop_1_moement_conv_3} as
		\begin{align}\label{eq:prop_1_moement_conv_5}
		\frac{\lfloor BC n \rfloor}{n} - z_{1,n+1} - (r/n)\ts z_{11} & \, = \,
		\E\left[(r/n)\ts X_{11}\land n^{\alpha}\ts -\ts (r/n)\ts z_{11}\right] \. + \. \E\left[ X_{1,n+1}\land n^{\alpha} - z_{1,n+1}\right]   \\
		&\qquad + \, \E\left[ \bigl( (r/n)\ts X_{11} - n^{\alpha} \bigr)^{+} \right] \. + \. \E\left[ \left( X_{1,n+1} - n^{\alpha} \right)^{+} \right].
		\end{align}
By Lemma~\ref{lemma:typical_corner}, the left hand side is of order $O(n^{\delta-1})$. The first term in the right hand side is of order $O(n^{\alpha + (1/2)-\delta+\eps})$ by Proposition \ref{prop:first_moment_truncation_bd}. Noting the bound in~\eqref{eq:prop_1_moement_conv_4} and that the last two terms in the right hand side are nonnegative, we deduce
\begin{align}
		\E\left[ \bigl( (r/n)\ts X_{11} - n^{\alpha} \bigr)^{+} \right] \, = \, O\bigl(n^{\delta-1}+n^{\alpha+(1/2)-\delta+\eps}\bigr),\quad\, \E\left[ \left( X_{1,n+1} - n^{\alpha} \right)^{+} \right] \, = \, O\bigl(n^{\delta-1}+n^{\alpha+(1/2)-\delta+\eps}\bigr).
\end{align}
This gives
\begin{align}
		n^{\delta-1} \ts\E\bigl[\ov{X}_{11}\bigr] \, = \, O\bigl( n^{\delta-1}+ n^{\alpha+(1/2)-\delta+\eps}\bigr),\qquad \E\left[ \ov{X}_{1,n+1} \right] \, = \,
        O\bigl( n^{\delta-1}+ n^{\alpha+(1/2)-\delta+\eps}\bigr).
\end{align}
		Since $\alpha,\eps>0$ are arbitrary, letting $\alpha\to 0$ proves the last equation in the proposition and completes the proof.
	\end{proof}

	We can now derive Theorem~\ref{thm:main_expectation} from Lemma~\ref{lemma:typical_corner}, Proposition~\ref{prop:first_moment_truncation_bd} and Proposition~\ref{prop:first_moment_error_bound}.
	
	\begin{proof}[\textbf{Proof of Theorem~\ref{thm:main_expectation}}]
		Let $\ov{X}=(\ov{X}_{ij})=X-Z$ denote the error table as before. From~\eqref{eq:limit_z_n+1,n+1} and the first equation in Proposition~\ref{prop:first_moment_error_bound}, we have
		\begin{align}
		\bigl|\E[X_{n+1,n+1}] - C\bigr| \,\le\, \bigl|\E[\ov{X}_{n+1,n+1}]\bigr| \. + \. |z_{n+1,n+1}-C| \,=\, O\bigl(n^{\delta-1}\bigr),
		\end{align}
		which holds for all $0\le \delta< 1$ and $B,C>0$. This shows (i).
		
		Next, we show (ii). Fix $\eps>0$. Suppose $0<\delta<1$ and $B<B_{c}$.  Then by  Proposition~\ref{prop:first_moment_error_bound} and Lemma~\ref{lemma:typical_corner}, we have
		\begin{align}
		\bigl|\E[X_{1,n+1}] - BC\bigr| \, \le \, \bigl|\E[\ov{X}_{1,n+1}]\bigr|
		\. + \. |z_{1,n+1}-BC| \, = \, O\bigl(n^{\delta-1} + n^{-(\delta/2)+\eps}\bigr).
		\end{align}
		Moreover, if $1/2<\delta<1$ and $B>B_{c}$, then similarly we have
		\begin{align}
		\bigl|\E[X_{1,n+1}] - B_{c}C\bigr| \, \le \, \bigl|\E[\ov{X}_{1,n+1}]\bigr|
		\. + \. |z_{1,n+1}-B_{c}C| \, = \, O\bigl(n^{\delta-1} + n^{(1/2)-\delta+\eps}\bigr).
		\end{align}
		Furthermore, if $0<\delta<1$ and $B>B_{c}$, then for every fixed $\alpha>0$, our Lemma~\ref{lemma:typical_corner} and Proposition~\ref{prop:first_moment_truncation_bd} give:
		\begin{align}
		\bigl|\E[X_{1,n+1}\land n^{\alpha}] - B_{c}C \bigr| \, \le \, \bigl|\E[X_{1,n+1}\land n^{\alpha}] - z_{1,n+1} \bigr|
		\. + \. |z_{1,n+1} - B_{c}C| \, = \, O\bigl(n^{\alpha-(\delta/2)+\eps} + n^{\delta-1}\bigr).
		\end{align}
		This shows~(ii).
		
		Finally, we obtain~(iii) as follows. Fix $\eps>0$. The first equation for $B<B_{c}$ follows
		from the second equation in Proposition~\ref{prop:first_moment_error_bound}.
		If we further assume that $1/2<\delta<1$, then by Proposition~\ref{prop:first_moment_truncation_bd}
		and Lemma~\ref{lemma:typical_corner}, we have:
		\begin{align}
		& \left| \E[X_{11}\land n^{\alpha}] \. -  \. \frac{B^{2}(1+C)}{(B_{c}-B)(B+B_{c}-2)}\right|
		\, \le\, \left| z_{11} \. - \. \frac{B^{2}(1+C)}{(B_{c}-B)(B+B_{c}-2)}   \right|\,+ \\
		& \hskip1.5cm \. + \. \bigl| \E[X_{11}\land n^{\alpha}] \. - \. z_{1,n+1}\bigr|  \, = \, O\bigl(n^{\delta-1} \ts + \ts n^{\alpha + (1/2)-\delta + \eps}\bigr).
		\end{align}
		Denote \ts $r=\lfloor n^{\delta} \rfloor$.  For $1/2<\delta<1$ and $B>B_{c}$, by Proposition~\ref{prop:first_moment_error_bound} and Lemma~\ref{lemma:typical_corner} we have:
		\begin{align}
		\bigl| (r/n)\ts\E[X_{11}]  - C(B-B_{c})\bigr| &\, \le\, \bigl| (r/n)\ts\E[X_{11}]  - (r/n)\ts z_{11}\bigr| \. + \.
		\bigl|(r/n)\ts z_{11}-C(B-B_{c})(r/n)\bigr|\\
		&\, = \. O\bigl(n^{\delta-1} + n^{\alpha+(1/2)-\delta+\eps}\bigr).
		\end{align}

		It remains to prove~(iii) for $0<\delta<1$ and $B>B_{c}$. Fix $\alpha>0$. Note that
		\begin{align}\label{eq:pf_thm_expectation_1}
		\frac{\lfloor BC n\rfloor}{n}  - B_{c}C \, = \, \frac{\lfloor n^{\delta} \rfloor}{n}\E[X_{11}]\. + \. \E[(X_{1,n+1}-n^{\alpha})^{+}] \. + \. (z_{1,n+1}-B_{c}C) + \E[X_{1,n+1}\land n^{\alpha}-z_{1,n+1}]\ts.
		\end{align}
		By Lemma~\ref{lemma:typical_corner} and Proposition~\ref{prop:first_moment_truncation_bd} (ii), the last two terms in the right hand side are of order $O(n^{\delta-1})$ and $O(n^{\alpha - (\delta/2) + \eps})$, respectively. Hence the difference between the left hand side and the sum of the first two terms in the right hand side must be of order $O(n^{\delta-1} +n^{\alpha - (\delta/2) + \eps} )$. This completes the proof of the theorem.
	\end{proof}

	\vspace{0.2cm}
	
	\section{Convergence of higher moments and law of large numbers}
	
	In this section, we prove Theorem~\ref{thm:SLLN}. Our argument is based on bounding the joint second moments of the entries in the uniform contingency table $X\in \mathcal{M}_{n,\delta}(B,C)$. We begin by generalizing Proposition~\ref{prop:first_moment_truncation_bd} into higher moments.

	\begin{prop}\label{prop:moment_truncation_bd}
		Fix $B,C>0 $ and denote $B_{c}=1+\sqrt{1+1/C}$. Let $X=(X_{ij})$ be sampled from $\M_{n,\delta}(B,C)$ uniformly at random. Let $Z=(z_{ij})$ be the typical table for $\mathcal{M}_{n,\delta}(B,C)$. Fix integers $k\ge 1$ and $1\le i_{\ell},j_{\ell}\le n+\lfloor n^{\delta} \rfloor$, $1\le \ell\le k$. Let $\eta(\delta)$ be defined as~\eqref{eq:def_eta_exponent}, and denote
		\begin{align}
		m_{11} = \text{$\#$ of $(i_{\ell},j_{\ell})$'s in the top left block of $\mathcal{M}_{n,\delta}(B,C)$}
		\end{align}
		Fix $\alpha,\eps>0$. Then unless both $m_{11}>0$ and $B=B_{c}$, for all sufficiently large $n\ge 1$, we have
		\begin{align}
		\left| \E\left[  \left( \prod_{\ell=1}^{k} X_{i_{\ell},j_{\ell}} \right)\land n^{\alpha + (1-\delta)m_{11}\mathbf{1}(B>B_{c})}\right] - \prod_{\ell=1}^{k} z_{i_{\ell},j_{\ell}} \right|
		\le n^{\alpha-\eta(\delta) + \eps}.
		\end{align}
	\end{prop}
	
	\begin{proof}
		Similar to the proof of Proposition~\ref{prop:first_moment_truncation_bd}. Details are omitted.
	\end{proof}

	\begin{prop}\label{prop:second_moment_bound}
		Fix $B,C> 0$ and $0\le \delta<1$ Let $X=(X_{ij})$ be sampled from $\M_{n,\delta}(B,C)$ uniformly at random. Let $Z=(z_{ij})$ be the typical table for $\M_{n,\delta}(B,C)$. Denote $(\ov{X}_{ij}) = X - Z$ and $r=\lfloor n^{\delta} \rfloor$. For $\delta>1/2$ and all $\eps>0$, we have:
		\begin{align}
		\E\left[ \left( \sum_{\ell=1}^{r}\ov{X}_{1\ell} \right)^{2} \right] \, + \, \E\left[ \left( \sum_{\ell=r+1}^{n+r}\ov{X}_{1\ell} \right)^{2} \right] \, + \, \left|\. 	 \E\left[ \left( \sum_{\ell=1}^{r}\ov{X}_{1\ell} \right)\left( \sum_{\ell=r+1}^{n+r}\ov{X}_{1\ell} \right) \right] \right| \, = \, O\bigl(n^{(5/2)-\delta+\eps}\bigr).
		\end{align}
		Furthermore, a similar bound holds for the $(n+1)$-th row for all $0\le \delta< 1$.
	\end{prop}
	
	\begin{proof}
		We only prove the first part of the proposition, for $1/2<\delta<1$. A similar argument shows the second part.	
		
		Taking square and expectation for the first equation in~\eqref{eq:margin_condition_Xbar} give
		\begin{align}
		\sum_{1\le i,j\le n+r} \E\left[ \ov{X}_{1i}\ov{X}_{1j} \right]  = 0.
		\end{align}
		Using the symmetry of the entries in each block, we can write the above equation as
		\begin{align}
		r \ts \E\left[ \ov{X}_{11}^{2} \right] \. + \. n \E\left[ \ov{X}_{1,n+1}^{2} \right] \.  + \.
        r(r-1) \ts \E\left[ \ov{X}_{11}\ov{X}_{12} \right] \. + \. n(n-1) \ts \E\left[ \ov{X}_{1,n+1}\ov{X}_{1,n+2} \right]\. +\. rn \ts
        \E\left[ \ov{X}_{11}\ov{X}_{1,n+1} \right]=0.
		\end{align}
		Denote $W_{1i} = r\ov{X}_{1i}/n$ for $1\le i \le r$. Then we have:
		\begin{align}
		\frac{1}{r} \E\left[ W_{11}^{2} \right]  \. + \. \frac{1}{n} \. \E\left[ \ov{X}_{1,n+1}^{2} \right] \. + \.
        \left(1- \frac{1}{r} \right) \ts \E\left[ W_{11}W_{12} \right] \. + \.
        \left(1- \frac{1}{n} \right) \ts \E\left[ \ov{X}_{1,n+1}\ov{X}_{1,n+2} \right]\. +\. \E\left[ W_{11}\ov{X}_{1,n+1} \right]\, = \, 0\ts.
		\end{align}

		Now fix $\alpha,\eps>0$ and denote $x^{+} = \max(x,0)$. Write
		\begin{align}\label{eq:prop_moement_conv_1}
		\E\left[ \ov{X}_{1,n+1}\ov{X}_{1,n+2}\right] \, = \, \E\left[ X_{1,n+1}X_{1,n+2}\land n^{\alpha}
        \ts - \ts z_{1,n+1}^{2}\right] \. + \. \E\left[ \left( X_{1,n+1}X_{1,n+2} \ts -\ts n^{\alpha} \right)^{+} \right]\ts.
		\end{align}
		Note that by Proposition~\ref{prop:moment_truncation_bd}, we have
		\begin{align}\label{eq:prop_moement_conv_2}
		\left| \E\left[ X_{1,n+1}X_{1,n+2}\land n^{\alpha} \. - \. z_{1,n+1}^{2}\right] \right| \, = \, O\bigl(n^{\alpha -(\delta/2)+\eps}\bigr).
		\end{align}
		Similarly, we can write
		\begin{align}
		\E\left[ W_{11}W_{12}\right] \. = \. \E\left[(rX_{11}/n)(rX_{12}/n)\land n^{\alpha} - (rz_{11}/n)^{2}\right] \. + \.
        \E\left[ \left( (rX_{11}/n)(rX_{12}/n) \ts - \ts n^{\alpha} \right)^{+} \right],
		\end{align}
		and Proposition~\ref{prop:moment_truncation_bd} gives
		\begin{align}
		\left| \E\left[(rX_{11}/n)(rX_{12}/n)\land n^{\alpha} \ts - \ts (rz_{11}/n)^{2}\right] \right| \, = \, O\bigl(n^{\alpha + (1/2)-\delta+\eps}\bigr).
		\end{align}
		By truncating the other terms and applying Proposition~\ref{prop:moment_truncation_bd} similarly, this gives
		\begin{align}
		&\frac{1}{r}\. \E\left[\left( (rX_{11}/n)^{2}-n^{\alpha} \right)^{+} \right] \. +\. \frac{1}{n}\. \E\left[ \left( X_{1,n+1}^{2}- n^{\alpha} \right)^{+}\right]
		\. + \. \left(1- \frac{1}{r}\right) \E\Bigl[ \bigl((rX_{11}/n)(rX_{12}/n)- n^{\alpha} \bigr)^{+}\bigr] \\
		&\quad + \left(1- \frac{1}{n}\right) \E\left[ \left( X_{1,n+1}X_{1,n+2}- n^{\alpha} \right)^{+}\right] + \E\left[ \left((rX_{11}/n) X_{1,n+1}- n^{\alpha} \right)^{+}\right]\,  = \, O\bigl(n^{\alpha + (1/2)-\delta+\eps}\bigr)\ts.
		\end{align}
		Since all terms in the left hand side of the above equation are nonnegative, they have order $O(n^{\alpha + (1/2)-\delta+\eps})$ individually. This give bounds on the truncation error for each moments.
		
		From~\eqref{eq:prop_moement_conv_1}, \eqref{eq:prop_moement_conv_2}, and the above bound on the truncation error, we deduce that
		\begin{align}
		\bigl|\E[\ov{X}_{1,n+1}\ov{X}_{1,n+2}]\bigr| \, = \, O\bigl(n^{\alpha+(1/2)-\delta+\eps}\bigr).
		\end{align}
		Since $\alpha,\eps>0$ were arbitrary, this gives:
		\begin{align}
		\bigl|\E[\ov{X}_{1,n+1}\ov{X}_{1,n+2}]\bigr| \,= \, O\bigl(n^{(1/2)-\delta+\eps}\bigr) \quad \text{for all} \ \eps>0\ts.
		\end{align}
		Similarly, we obtain:
		\begin{align}
		&\bigl|\E[W_{11}^{2}]\bigr| \, = \, O\bigl(n^{(1/2)+\eps}\bigr), \qquad \
		\bigl|\E[\ov{X}_{1,n+1}^{2}]\bigr| \, = \, O\bigl(n^{(3/2)-\delta+\eps}\bigr), \\
		&\bigl|\E[W_{11}W_{12}]\bigr| \. + \. \bigl|\E[X_{11}\ov{X}_{1,n+1}]\bigr| \, = \, O\bigl(n^{(1/2)-\delta+\eps}\bigr).
		\end{align}
		The result now follows by combining these bounds.
	\end{proof}

	
	\begin{proof}[\textbf{Proof of Theorem~\ref{thm:SLLN}}]
		We only prove the first part of the theorem, for $1/2<\delta<1$. A similar argument shows the second part.	
		
		Let $Z=(z_{ij})$ be the typical table for $\M_{n,\delta}(B,C)$. Let $S_{n,\delta}(B,C)$ denote the sum of the entries in the first row of the up right block, as defined in Conjecture \ref{conj:CLT}. According to Lemma~\ref{lemma:typical_corner}, as $n\rightarrow \infty$,
		\begin{align}
		z_{1,n+1} \rightarrow
		\begin{cases}
		BC & \text{if $B<B_{c}$} \\
		B_{c}C & \text{if $B>B_{c}$}.
		\end{cases}
		\end{align}
		Hence by the triangle inequality, it suffices to show that
		\begin{align}\label{eq:thm_SLLN_claim}
		\left| n^{-1}S_{n,\delta}(B,C) - z_{1,n+1} \right| \. \rightarrow \. 0 \quad \ \text{as \. $n\rightarrow \infty$, \ts a.s.}
		\end{align}

		Denote \ts $r=\lfloor n^{\delta} \rfloor$ \ts and \ts $(\ov{X}_{ij}) = X - Z$. Let
		\begin{align}
		\ov{S}_{n,\delta}(B,C) \, := \, S_{n,\delta}(B,C) \. - \. n\ts z_{1,n+1} \, = \, \ov{X}_{1,r+1} + \ov{X}_{1,r+2} + \ldots+ \ov{X}_{1,r+n}\..
		\end{align}
		By the Markov inequality, we have:
		\begin{align}
		\P\left( \ov{S}_{n,\delta}(B,C) \ge t \right) \, \le \, \frac{1}{t^2} \. \E\left[ \left( \sum_{k=r+1}^{n+r}\ov{X}_{1k} \right)^{2} \right].
		\end{align}
		Hence, by Proposition~\ref{prop:second_moment_bound}, we have that for every $\xi,\eps>0$ there exists a
		constant $c>0$, s.t.
		\begin{align}
		\P\left( \ov{S}_{n,\delta}(B,C) \ge n^{1-\xi} \right) \. \le \. c \ts n^{(1/2)-\delta+2\xi+\eps}
		\end{align}
		for all sufficiently large $n\ge 1$. Thus, if we choose \ts $0<\xi<1/2 - \delta$, then for some \ts $\xi',c'>0$ \ts we have:
		\begin{align}
		\P\left( \ov{S}_{n,\delta}(B,C) \ge n^{1-\xi} \right) \. \le \. c' n^{-\xi'}
		\end{align}
		for all sufficiently large $n\ge 1$. This implies that for every sequence \ts $(n_{k})_{k\ge 1}$ \ts s.t.\
		$n_{k}\rightarrow \infty$ as $k\rightarrow \infty$, we can choose a subsequence $n_{k(r)}\rightarrow \infty$
		as $r\rightarrow \infty$, s.t.\
		\begin{align}
		\sum_{r=1}^{\infty} \,\ts \P\left( \frac{\ov{S}_{n_{k(r)},\delta}(B,C)}{n_{k(r)}}\. \ge \. (n_{k(r)})^{-\xi} \right) \. < \.\infty\ts.
		\end{align}
		By the Borel--Cantelli lemma, it follows that
		\begin{align}
		\frac{\ov{S}_{n_{k(r)},\delta}(B,C)}{n_{k(r)}} \rightarrow 0 \quad \text{\. as \, $r\rightarrow \infty$, \. a.s.}
		\end{align}
		Therefore, almost surely,
		\begin{align}
		\liminf_{n\rightarrow \infty} \. n^{-1}\ov{S}_{n,\delta}(B,C) \, = \,
		\limsup_{n\rightarrow \infty} \. n^{-1}\ov{S}_{n,\delta}(B,C) \, = \, 0\ts.
		\end{align}
		This shows~\eqref{eq:thm_SLLN_claim}, and completes the proof.
	\end{proof}

	\vspace{0.3cm}
	
	\medskip
	
	\section{Final remarks and open problems}
	\label{s:finrem}
	
	\subsection{Large $C$ limit}  In the limiting case $C\to \infty$ our results
	give continuous distributions, which can be viewed
	as results on projections of transportation polytopes $\mathcal{P}(\r,\c)$,
	with $\r=\c$ given by~\eqref{eq:def_barvinok_margin} for $C=1$.  Scaled
	properly, in the limit the geometric distribution becomes the exponential
	distribution, so we recover the Theorem~1 in~\cite{CDS} from our Theorem~\ref{thm:main_geo}.
	
	Note that in the continuous case the critical value
	$$B_c \. = \, \lim_{C\to \infty} \. 1+ \sqrt{1+1/C} \. = \. 2\ts.
	$$
	In~\cite{DP}, the authors tested the conjectures for the larger values of~$C$.
	Figure~\ref{fig:Barv-sim-C} is similar to Figure~\ref{fig:Barv-sim} with $C=9$,
	with a subcritical $B=1.8$ and supercritical $B=2.2$, where \ts
	$B_c=1+\sqrt{1+1/9}\approx 2.05$.
	Here fewer trials were used as the algorithm slows down for larger~$C$, so the
	graphs are visibly less smooth.  Again, the normal distribution is evident
	in the supercritical case.
	
	\begin{figure}[hbt]
		\begin{center}
			\includegraphics[width = 0.36 \textwidth]{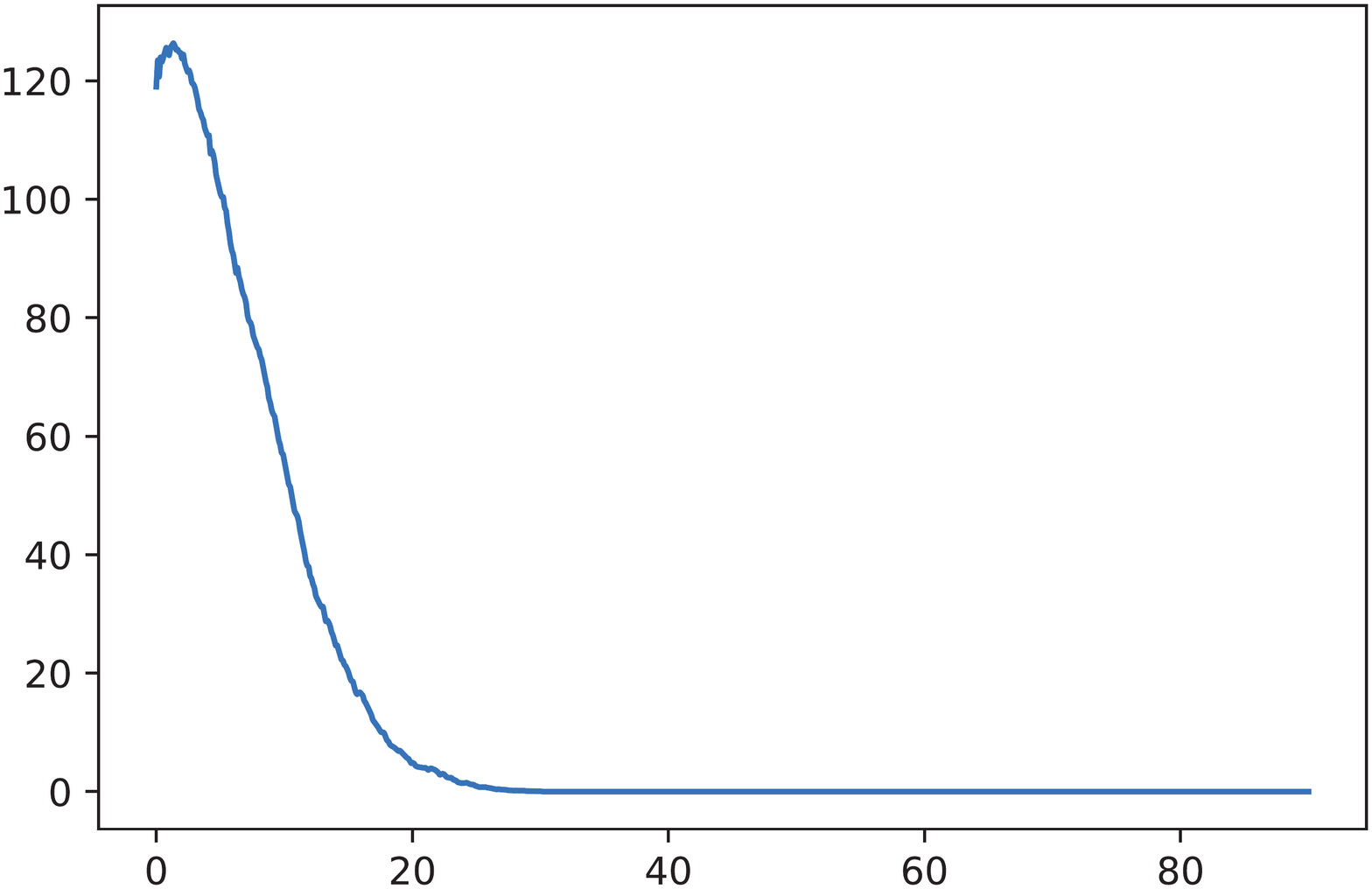} \qquad \ \
			\includegraphics[width = 0.36 \textwidth]{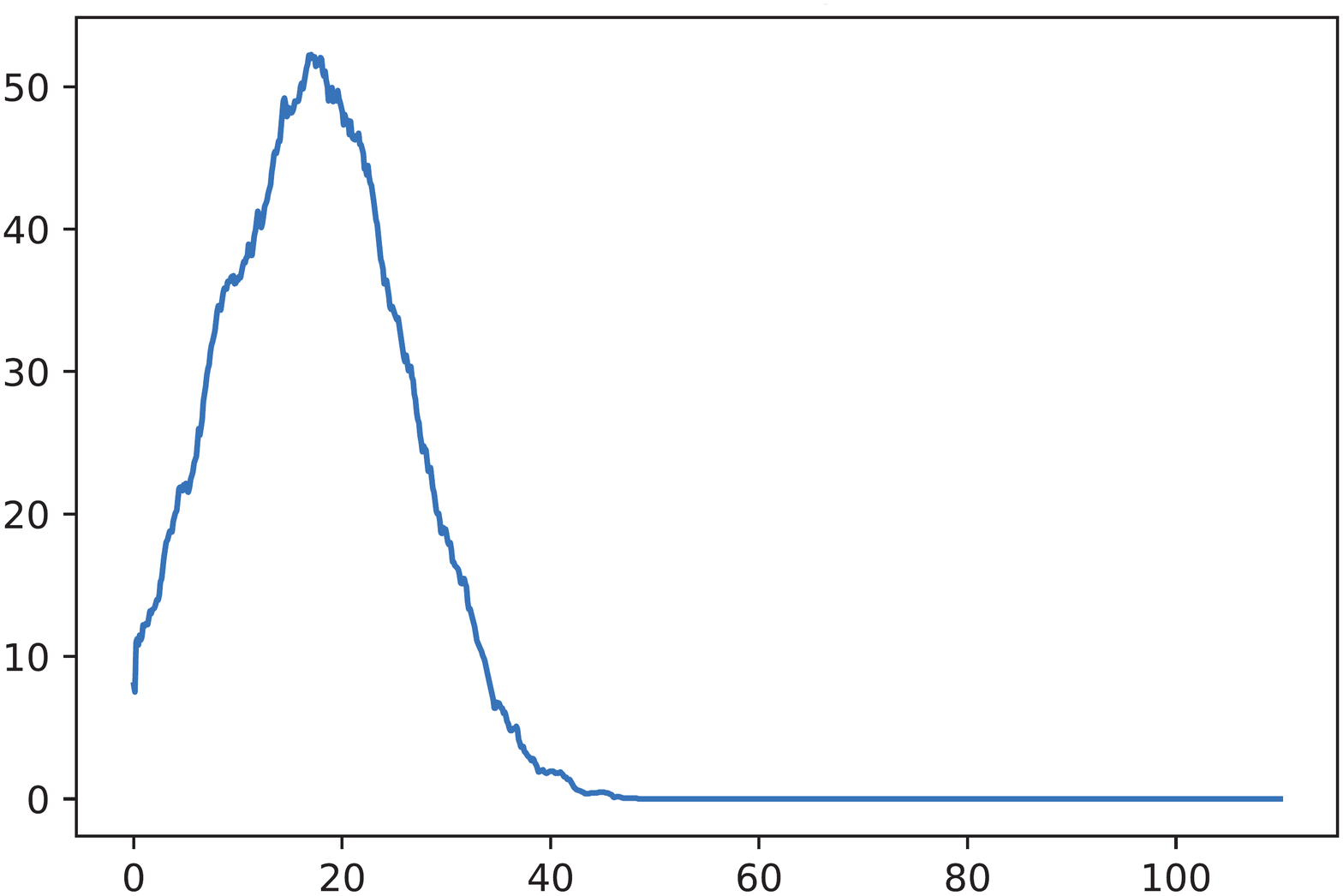}	
		\end{center}
		\caption{Summary of $10^3$ simulations of $x_{11}/C$ in a random uniform $n\times n$ contingency table $X=(x_{ij})$
			with both margins \ts $(BC\ts n,C\ts n,\ldots,C\ts n)$, where $n=50$ and $C=9$.  The left graph corresponds to a subcritical value $B=1.8$,
			and the right graph to a supercritical value $B=2.2$.}
		\label{fig:Barv-sim-C}
	\end{figure}

	\smallskip
	
\subsection{Single special row and column case}
As we mentioned in the introduction, this paper was motivated by Barvinok's
question about the phase transition for a single special row and column
(the case $\de=0$ in our notation).  We are now confident in the conjecture
that the phase transition occurs at $B=B_c$ based on our results and
extensive simulations, see Figure~\ref{fig:Barv-sim} and~\cite{DP}.
In fact, we believe both Conjecture~\ref{conj:CLT} and~\ref{conj:Gaussian}
extend to the single row and column case.  Unfortunately, this case
remains out of reach by the available tools.

	\smallskip
	
\subsection{Symmetric, binary, and high-dimensional tables}
There are several variations on results in this paper
one can consider.  First, one can study \emph{symmetric} contingency
tables $X=X^T$ with the same row and column sums as in~\eqref{eq:def_barvinok_margin}.
These correspond to adjacency matrices of multigraphs with degree vectors~$\r$.
The results in this paper extend verbatim to this setting.
	
It would be even more interesting to understand the cases of general and
symmetric binary (0-1) contingency tables (with zeros on the diagonal
in the symmetric case),
which correspond to simple bipartite and usual graphs.  This is an extremely well
studied family of problems, see~\cite{W2} for a recent thorough survey
and~\cite{B2,BH1} for the typical tables approach in this case.  Note that
since the entries are restricted to 0-1, the resulting distributions are
Bernoulli but one needs to be careful in describing the phase transition.
We believe our $\de >1/2$ results should be possible to modify in this case.
	
It would be also interesting to see if our results extend to
higher dimensional contingency tables, see Remark~\ref{r:high-dim}.
We reserve the judgement in this case; cf.~\cite{DP0} however.
	
	\smallskip
	
\subsection{Criticality}  The phase transition exhibited in the paper is somewhat
different from known examples of phase transition in Statistical
Mechanics, see~\cite{Ye}, Random Graph Theory, see~\cite{AS,Bol},
or MCMC, see~\cite{Dia}. It would be very interesting to understand
exactly what happens at the critical values $B=B_c$ or $\de=1/2$.
Unfortunately, these problems remain out of reach both theoretically and
experimentally.

	\vskip.5cm
	
\subsection*{Acknowledgements}
We are grateful to Sasha Barvinok for telling us about the problem
and for the numerous interesting discussions.  The third author was partially
supported by the NSF.

	\vspace{1cm}
\addresseshere

\newpage

\appendix

\vspace{0.2cm}

\section{Asymptotic independence between entries}	

In this appendix, we show that any two entries in the bottom right block of $\mathcal{M}_{n,\delta}(B,C)$ are asymptotically independent in the sense that the correlation between any two entries of the uniform contingency table $X$ decays polyonimally depending on the blocks that the entries belong. 

\begin{theorem}\label{thm:asymp_indep}
	Fix $\delta\in [0,1]$, $B,C\ge 0$,  and let $X=(X_{ij})$ be sampled from $\M_{n,\delta}(B,C)$ uniformly at random. Fix $\eps>0$. Then for every $\eps\ge 1$, we have  
	\begin{align}
	& \sup_{a,b\in \mathbb{N}} \left| \P(X_{i_{1},j_{1}}=a,\, X_{i_{2},j_{2}}=b) - \P(X_{i_{1},j_{1}}=a) \P(X_{i_{2},j_{2}}=b) \right| \le n^{-\eta(\delta)+\eps}
	\end{align}
	for all sufficiently large $n\ge 1$, where 
	\begin{align}\label{eq:def_eta_exponent}
	\eta(\delta) =
	\begin{cases}
	1/2 & \text{if all $(i_{r},j_{r})$'s are contained in the bottom right block,} \\
	\delta-1/2  & \text{if some $(i_{r},j_{r})$ is contained in the top left block,} \\
	\delta/2 & \text{otherwise}.
	\end{cases}
	\end{align}
\end{theorem}	

The following lemma is a minor generalization of Lemma \ref{lemma:key}.

\begin{lemma}[Concentration of submatrices]\label{lemma:key2}
	Let $\mathcal{M}(\r,\c)$ be the set of all \ts $m\times n$ contingency tables of margins $\r$ and~$\c$. Let $X=(X_{ij})$ be sampled from $\mathcal{M}(\r,\c)$ uniformly at random. Suppose $\mathcal{B}_{1},\ldots,\mathcal{B}_{k}$ are blocks in $\mathcal{M}(\r,\c)$ with $|\mathcal{B}_{1}|\le \ldots \le |\mathcal{B}_{k}|$. Then there exists an absolute constant $\gamma>0$, s.t.\ for each $I=(I_1,\ldots,I_k)\in \mathcal{B}_{1}\times \ldots \times \mathcal{B}_{k}$ and $t>0$, we have:
	\begin{align}
	d_{TV}\left( (X_{I_{\ell}})_{\ell=1}^{k}, (Y_{I_{\ell}})_{\ell=1}^{k}\right) \le t + N^{\gamma(m+n)} \exp\left( - \left\lfloor \frac{|\mathcal{B}_{1}|}{k} \right\rfloor \frac{t^{2}}{2} \right),
	\end{align}
	where $Z=(z_{ij})$ is the typical table for \ts $\mathcal{M}(\r, \c )$, and $Y=(Y_{J})$ is the random matrix of independent entries with \ts
	$Y_{ij}\sim \Geom(z_{ij})$, and \ts $N=r_1+\ldots+r_m = c_1+\ldots+c_n$.
\end{lemma}

\begin{proof}
	Let \ts $\M(\r, \c)$ \ts blocks \ts $\mathcal{B}_{1},\ldots, \mathcal{B}_{k}$ \ts such that \ts $|\mathcal{B}_{1}|\le \ldots \le |\mathcal{B}_{k}|$. Let $Z=(z_{ij})$ be the typical table for \ts $\mathcal{M}(\r,\c)$, and let \ts $Y=(Y_{J})$ \ts denote the random matrix of independent entries where $Y_{ij}\sim \Geom(z_{ij})$. Observe that we can choose a subset $\mathcal{I}\subseteq \mathcal{B}_{1}\times \ldots \times \mathcal{B}_{k}$ such that $|\mathcal{I}|\ge |\mathcal{B}_{1}|/k$ and every two elements of $\mathcal{I}$ have distinct coordinates. Fix measurable sets \ts $A\subseteq \R_{+}^{k}$ \ts and \ts
	$\mathcal{A}\subseteq \R_+^{m n}$.
	
	For a \ts $m\times n$ matrix $W=(w_{ij})$ \ts and \ts $I\in \mathcal{B}_{1}\times \ldots \times \mathcal{B}_{k}$, denote
	\begin{align}
	W^{I} \, :=  (w_{I_\ell})_{\ell=1}^{k}\in \R^{k}, \qquad S(W) \, := \, \frac{1}{|\mathcal{I}|}\. \sum_{I\in \mathcal{I}} \, \mathbf{1}_{\{W^{I} \in A\}}\,.
	\end{align}
	Note that by the exchangeability of the entries of $X$ and $Y$ in each block of \ts $\mathcal{M}(\r,\c)$, random vectors \ts
	$X^{I}$ \ts and \ts $Y^{I}$ \ts have the same distribution for all~$I$. In particular, we have:
	\begin{align}
	\P(X^{I}\in A) \. = \. \E\bigl[S(X)\bigr]\ts.
	\end{align}
	Moreover, since $Y_{ij}$ are independent and since every two elements of \ts $\mathcal{I}$ \ts
	have non-overlapping coordinates, it follows that \ts $Y^{I}$ \ts are also independent.
	
	Now the rest of the argument is identical to the proof of Lemma \ref{lemma:key}.
\end{proof}

The following is a vector version of Theorem \ref{thm:geo_approximation}.

\begin{theorem}\label{thm:geo_approximation_vector}
	Let $X=(X_{ij})$ be sampled from $\M_{n,\delta}(B,C)$ uniformly at random. Let $Z=(z_{ij})$ be the typical table for $\mathcal{M}_{n,\delta}(B,C)$ and let $Y=(Y_{ij})$ denote the random matrix of independent entries with $Y_{ij}\sim \Geom(z_{ij})$. Fix integers $k\ge 1$ and $1\le i_{r},j_{r}\le n+\lfloor n^{\delta} \rfloor$, $1\le r\le k$. Then for every $\eps>0$, we have:
	\begin{align}
	d_{TV}\left( (X_{i_{r},j_{r}})_{r=1}^{k}, ( Y_{i_{r},j_{r}})_{r=1}^{k}\right)
	\le n^{-\eta(\delta) + \eps}
	\end{align}
	for all sufficiently large $n\ge 1$, where $\eta(\delta)$ is defined as in \eqref{eq:def_eta_exponent}. 
\end{theorem}

\begin{proof}
	Identical to that for Theorem \ref{thm:geo_approximation} using Lemma \ref{lemma:key2}.
\end{proof}

\begin{proof}[\textbf{Proof of Theorem \ref{thm:asymp_indep}}]
	Let $Z=(z_{ij})$ be the typical table for \ts $\mathcal{M}(\r,\c)$, and let \ts $Y=(Y_{J})$ \ts denote the random matrix of independent entries where $Y_{ij}\sim \Geom(z_{ij})$. By the symmetry within blocks, in order to bound the maximum in the assertion, we only need to consider six cases corresponding to the combination of blocks that entries $X_{i,j}$ and $X_{k,\ell}$ belong. 
	
	Consider the case when both entries are contained in the top left blocks. The other cases can be shown by a similar argument.  Note that for any $a,b\in \mathbb{N}$,
	\begin{align}
	&\left| \P(X_{n+1,n+1}=a,\, X_{n+2,n+2}=b) - \P(Y_{n+1,n+1}=a,\, Y_{n+2,n+2}=b) \right| \\
	&\qquad \le d_{TV}((X_{n+1,n+1},X_{n+2,n+2}), (Y_{n+1,n+1},Y_{n+2,n+2})), \\
	&\left| \P(X_{n+1,n+1}=a) - \P(Y_{n+1,n+1}=a) \right|  \le d_{TV}(X_{n+1,n+1}, Y_{n+1,n+1}), \\
	&\left| \P(X_{n+2,n+2}=b) - \P(Y_{n+2,n+2}=b) \right|  \le d_{TV}(X_{n+2,n+2}, Y_{n+2,n+2}).
	\end{align}
	Since $Y_{n+1}$ and $Y_{n+2}$ are independent, triangle inequality gives 
	\begin{align}
	&\sup_{a,b\in \mathbb{N}} \left| \P(X_{n+1,n+1}=a,\, X_{n+2,n+2}=b) - \P(X_{n+1,n+1}=a) \P(X_{n+2,n+2}=b) \right| \\
	& \qquad \le d_{TV}((X_{n+1,n+1},X_{n+2,n+2}), (Y_{n+1,n+1},Y_{n+2,n+2})) + d_{TV}(X_{n+1,n+1}, Y_{n+1,n+1}) + d_{TV}(X_{n+2,n+2}, Y_{n+2,n+2}).
	\end{align}
	Then by Theorem \ref{thm:geo_approximation}, for any fixed $\eps>0$, each of the total variation distances above is at most $n^{-1/2+\eps}$ for sufficiently large $n\ge 1$. 
\end{proof}


\vskip.9cm
	
{\footnotesize
	
	\begin{thebibliography}{BLSY10}
		
		\bibitem[AS16]{AS}
		N.~Alon and J.~H.~Spencer, \emph{The probabilistic method} (Fourth~ed.),
		John Wiley, Hoboken, NJ, 2016.
		
		\bibitem[Bar09]{B1}
		A.~Barvinok,
		Asymptotic estimates for the number of contingency tables, integer flows, and volumes of transportation polytopes,
		\emph{Internat.\ Math.\ Res.\ Notices}~\textbf{2009} (2009), 348--385.
		
		\bibitem[Bar10a]{B2}
		A.~Barvinok, On the number of matrices and a random matrix with prescribed
		row and column sums and 0-1 entries,
		\emph{Adv.\ Math.}~\textbf{224} (2010), 316--339.
		
		\bibitem[Bar10b]{B3}
		A.~Barvinok,
		What does a random contingency table look like?,
		\emph{Comb.\ Probab.\ Comp.}~\textbf{19} (2010), 517--539.
		
		\bibitem[Bar12]{B4}
		A.~Barvinok,
		Matrices with prescribed row and column sums,
		\emph{Lin.\ Alg.\ Appl.}~\textbf{436} (2012), 820--844.
		
		\bibitem[BH10]{BH}
		A.~Barvinok and J.~A.~Hartigan, Maximum entropy gaussian
		approximations for the number of integer points and volumes of polytopes,
		\emph{Adv.\ Applied Math.}~\textbf{45} (2010), 252--289.
		
		\bibitem[BH13]{BH1}
		A.~Barvinok and J.~A.~Hartigan,
		The number of graphs and a random graph with a given degree sequence,
		\emph{Rand.\ Struct.\ Algor.}~\textbf{42} (2013), 301--348.
		
		\bibitem[BLSY10]{BLSY}
		A.~Barvinok, Z.~Luria, A.~Samorodnitsky and A.~Yong,
		An approximation algorithm for counting contingency tables,
		\emph{Rand.\ Struct.\ Algor.}~\textbf{37} (2010), 25--66.
		
		\bibitem[BP03]{BP}
		M.~Beck and D.~Pixton,
		The Ehrhart polynomial of the Birkhoff polytope,
		\emph{Discrete Comput.\ Geom.}~\textbf{30} (2003), 623--637.
		
		\bibitem[BBK72]{BBK}
		A.~B\'{e}k\'{e}ssy, P.~B\'{e}k\'{e}ssy and J.~Koml\'{o}s,
		Asymptotic enumeration of regular matrices,
		\emph{Studia Sci.\ Math.\ Hungar.}~\textbf{7} (1972), 343--353.
		
		\bibitem[BC78]{BC}
		E.~A.~Bender and E.~R.~Canfield, The asymptotic number of labeled graphs with given degree sequences,
		\emph{J.~Combin.\ Theory, Ser.~A}~\textbf{24} (1978), 296--307.
		
		\bibitem[Bol01]{Bol}
		B.~Bollob\'as, \emph{Random graphs} (Second ed.),  Cambrie Univ.\ Press, Cambridge, UK, 2001.
		
		\bibitem[CM09]{CM-birk}
		E.~R.~Canfield and B.~D.~McKay,
		The asymptotic volume of the Birkhoff polytope,
		\emph{Online J.~Anal.\ Comb.}~\textbf{4} (2009), Paper~2, 4~pp.
		
		\bibitem[CM10]{CM}
		E.~R.~Canfield and B.~D.~McKay,
		Asymptotic enumeration of integer matrices with large equal row and column sums,
		\emph{Combinatorica}~\textbf{30} (2010), 655--680.
		
		
		\bibitem[CDS10]{CDS}
		S.~Chatterjee, P.~Diaconis and A.~Sly, Properties of uniform doubly
		stochastic matrices, preprint (2010), 18 pp.; {\tt arXiv:1010.6136}.
		
		\bibitem[CV16]{CV}
		B.~Cousins and S.~Vempala,
		A practical volume algorithm,
		\emph{Math.\ Program.\ Comput.}~\textbf{8} (2016), 133--160.
		
		\bibitem[DK14]{DK}
		J.~De Loera and E.~D.~Kim,
		Combinatorics and geometry of transportation polytopes: an update, in
		\emph{Discrete geometry and algebraic combinatorics}, AMS,
		Providence, RI, 2014, 37--76.
		
		\bibitem[Dia96]{Dia}
		P.~Diaconis,
		The cutoff phenomenon in finite Markov chains,
		\emph{Proc.\ Nat.\ Acad.\ Sci.\ U.S.A.}~\textbf{93} (1996), no.~4, 1659--1664.
		
		
		\bibitem[DE85]{DE}
		P.~Diaconis and B.~Efron,  Testing for independence in a two-way table:
		new interpretations of the chi-square statistic,
		\emph{Ann.\ Stat.}~\textbf{13} (1985), 845--913.
		
		\bibitem[DG95]{DG}
		P.~Diaconis and A.~Gangolli,
		Rectangular arrays with fixed margins,
		\emph{Disc.\ Prob.\ Alg.}~\textbf{72} (1995), 15--41.
		
		\bibitem[DS98]{DS}
		P.~Diaconis and B.~Sturmfels,
		Algebraic algorithms for sampling from conditional distributions,
		\emph{Ann.\ Stat.}~\textbf{26} (1998), 363--397.
		
		\bibitem[DP18]{DP0}
		S.~Dittmer and I.~Pak, Random sampling and approximate counting of sparse
		contingency tables, submitted (2018).
		
		\bibitem[DLP19+]{DP}
		S.~Dittmer, H.~Lyu, and I.~Pak, Contingency tables have empirical bias, in
		preparation (2019).
		
		\bibitem[Eve92]{Ev}
		B.~S.~Everitt,
		\emph{The analysis of contingency tables} (Second ed.),
		Chapman \& Hall, London, 1992.
		
		\bibitem[FLL17]{FLL}
		M.~W.~Fagerland, S.~Lydersen and P.~Laake,
		\emph{Statistical analysis of contingency tables},
		CRC Press, Boca Raton, FL, 2017.
		
		\bibitem[Goo63]{Good}
		I.~J.~Good, Maximum entropy for hypothesis formulation, especially for
		multidimensional contingency tables,
		\emph{Annals Math.\ Stat.}~\textbf{34} (1963), 911--934.
		
		\bibitem[Goo76]{Good76}
		I.~J.~Good, On the application of symmetric Dirichlet distributions and their mixtures to contingency tables,
		\emph{Annals Stat.}~\textbf{4} (1976), 1159--1189.
		
		\bibitem[GM08]{GM}
		C.~Greenhill and B.~D.~McKay,
		Asymptotic enumeration of sparse nonnegative integer matrices with specified
		row and column sums, \emph{Adv.\ Appl.\ Math.}~\textbf{41} (2008), 459--481.
		
		\bibitem[Kat14]{Kat}
		M.~Kateri,
		\emph{Contingency table analysis},
		Birkh{\"a}user, New York, NY, 2014.
		
		\bibitem[Pak00]{Pak}
		I.~Pak,
		Four questions on Birkhoff polytope,
		\emph{Ann.\ Comb.}~\textbf{4} (2000), 83--90.
		
		\bibitem[Wor18]{W2}
		N.~Wormald,
		Asymptotic enumeration of graphs with given degree sequence, in
		\emph{Proc.\ ICM Rio de Janeiro}, Vol.~3, 2018, 3229--3248.
		
		\bibitem[Yep18]{Ye}
		J.~M.~Yeomans, \emph{Statistical Mechanics
			of Phase Transitions}, Oxford Univ.\ Press, 1992.
		
		
	\end{thebibliography}
	
}
\end{document}